\documentclass[11pt]{article}
\usepackage{hyperref}
\usepackage{authblk}
\usepackage{amsmath,bbm}
\usepackage{amssymb}
\usepackage{amsthm}
\usepackage{amscd}
\usepackage{amsfonts}
\usepackage{graphicx}
\usepackage{fancyhdr}
\usepackage{tikz}
\usepackage{multicol}
\usepackage{enumitem}

\usepackage[numbers]{natbib}

\usetikzlibrary{cd}
\usetikzlibrary{positioning,decorations.pathmorphing}
\usetikzlibrary{arrows.meta}

\newtheorem{theorem}{Theorem}[section]

\newtheorem{corollary}[theorem]{Corollary}

\newtheorem{lemma}[theorem]{Lemma}
\newtheorem{prop}[theorem]{Proposition}

\newtheorem*{theorem*}{Theorem}

\theoremstyle{definition}

\newtheorem{ex}[theorem]{Example}

\newtheorem{defn}[theorem]{Definition}

\newtheorem{remark}[theorem]{Remark}

\title{Thin Posets, CW Posets, and Categorification}
\author{Alex Chandler}
\affil{Faculty of Mathematics, University of Vienna\\
alex.chandler@univie.ac.at}

\newcommand{\mc}[1]{\mathcal{#1}}

\newcommand{\im}{\textnormal{im}\,}
 
\newcommand{\simp}{\textnormal{simp}}

\newcommand{\gr}{\textnormal{gr}}

\newcommand{\Int}{\textnormal{Int}}

\newcommand{\cell}{\textnormal{cell}}

\newcommand{\RGF}{\textnormal{RA}} 
 
\newcommand{\Ob}{\textnormal{Ob}}

\newcommand{\op}{\textnormal{op}}
\newcommand{\Br}{\textnormal{Br}} 
 
\newcommand{\Hom}{\textnormal{Hom}} 
\newcommand{\rk}{\textnormal{rk}}
\newcommand{\qrk}{q\textnormal{rk}}
\newcommand{\N}{\mathbb{N}} 
\newcommand{\R}{\mathbb{R}} 
\newcommand{\Z}{\mathbb{Z}}

\newcommand{\pinch}{\hspace{.5mm}\usebox\opinch\hspace{.5mm}}

\newcommand{\tvcplx}{\mathcal{C}^*(T_{2,n},\vec{x})}

\newcommand{\inv}{\textnormal{inv}}

\newcommand{\norm}[1]{\left\lVert #1 \right\rVert}

\newcommand{\Rmod}{R\textnormal{-}\mathbf{mod}} 
\newcommand{\PoCo}{\mathbf{PoCo}}
\newcommand{\PoHo}{\mathbf{PoHo}}

\newcommand{\Zmod}{\Z\textbf{-mod}}

\newcommand{\Rgmod}{R\textbf{-gmod}}

\newcommand{\Cabel}{\mc{C}(\mathbf{Abel})}
\newcommand{\Abel}{\mathbf{Abel}}
\newcommand{\grAbel}{\textbf{gr}\mathbf{Abel}}

\definecolor{forest}{rgb}{0.03, 0.47, 0.19}

\newsavebox\etatrans
\begin{lrbox}{\etatrans}

\end{lrbox}

\newsavebox\cattable
\begin{lrbox}{\cattable}
\begin{tabular}{ |c|c| } 
 \hline
 sets &  categories \\ 
 \hline
 elements &  objects \\ 
 \hline
 equations between &  isomorphisms between \\ 
 elements&objects\\
 \hline
 functions &  functors \\ 
 \hline
 equations between &  natural transformations \\ 
 functions&between functors\\
 \hline
\end{tabular}%
\end{lrbox}

\newsavebox\linearextension
\begin{lrbox}{\linearextension}
\raisebox{-2mm}{\scalebox{.3}{\usebox\graphthreefour}} 
$\leq $
\raisebox{-2mm}{\scalebox{.3}{\usebox\graphthreefourfive}} 
$\leq $
\raisebox{-2mm}{\scalebox{.3}{\usebox\graphonetwo}} 
$\leq $
\raisebox{-2mm}{\scalebox{.3}{\usebox\graphonetwofive}} 
$\leq $
\raisebox{-2mm}{\scalebox{.3}{\usebox\graphtwothreefour}} 
$\leq $
\raisebox{-2mm}{\scalebox{.3}{\usebox\graphtwothreefourfive}} 
$\leq $
\raisebox{-2mm}{\scalebox{.3}{\usebox\graphonethreefour}} 
$\leq $
\raisebox{-2mm}{\scalebox{.3}{\usebox\graphonethreefourfive}} 
$\leq $
\raisebox{-2mm}{\scalebox{.3}{\usebox\graphonetwofour}} 
$\leq $
\raisebox{-2mm}{\scalebox{.3}{\usebox\graphonetwofourfive}} 
$\leq $
\raisebox{-2mm}{\scalebox{.3}{\usebox\graphonetwothree}} 
$\leq $
\raisebox{-2mm}{\scalebox{.3}{\usebox\graphonetwothreefive}} 
$\leq $
\raisebox{-2mm}{\scalebox{.3}{\usebox\graphonetwothreefourfive}}
\end{lrbox}

\newsavebox\bcexample
\begin{lrbox}
{\bcexample}

\end{lrbox}

\usepackage{wrapfig,comment}
\begin{document}
\maketitle

\begin{abstract}
Motivated by Khovanov's categorification of the Jones polynomial, we study functors $F$ from thin posets $P$ to abelian categories $\mc{A}$. Such functors $F$ produce cohomology theories $H^*(P,\mc{A},F)$. We find that CW posets satisfy conditions making them particularly suitable for the construction of such cohomology theories. We consider a category of tuples $(P,\mc{A},F,c)$, where $c$ is a certain $\{1,-1\}$-coloring of the cover relations in $P$, and show the cohomology arising from a tuple $(P,\mc{A},F,c)$ is functorial, and independent of the coloring $c$ up to natural isomorphism. 
Such a construction provides a framework for the categorification of a variety of familiar topological/combinatorial invariants: anything expressible as a rank-alternating sum over a thin poset. 
\end{abstract}

\tableofcontents

\section{Introduction}
\label{chap-two}

In 2000, Khovanov \cite{khovanov1999categorification} categorified the Jones polynomial by lifting Kauffman's state sum formula (a rank-alternating sum over a Boolean lattice) to a homology theory built from a functor on the Boolean lattice. This construction was further explained and popularized soon after by Bar-Natan \cite{bar2002khovanov} and Viro \cite{viro2004khovanov}. Since Khovanov's categorification, many authors have used similar techniques to categorify other polynomials which admit state sum formulas over Boolean lattices (see for example \cite{dancso2015odd,  helme2005categorification, hepworth2015categorifying,  sazdanovic2018categorification, stovsic2008categorification}). In this paper, we set out a general framework for extending this technique to categorify objects which admit state sum formulas over a larger class of posets, called thin posets, containing Boolean lattices as a special case. Thin posets are defined as ranked finite posets whose intervals of length two are isomorphic to the `diamond' poset \usebox\diamondddd \ with elements $\{a,b,c,d\}$ and cover relations $\{a\lessdot b, b\lessdot d, a\lessdot c, c\lessdot d\}$.

This paper will be organized as follows. In Section \ref{posetbackground}, we review some basic poset terminology, and give a wealth of examples of thin posets. In Section \ref{dtsec} we introduce the class of diamond transitive thin posets. Diamond transitive thin posets $P$ have the property that there are far fewer conditions needed to check that a functor on $P$ is well defined. In particular, for diamond transitive thin posets, one need only show that compositions of morphisms agree along any two coinitial coterminal directed paths of length 2 in the Hasse diagram of $P$, whereas in general one must check this condition for chains of all lengths. The most important example of thin posets appearing in this paper are the CW posets, that is, the face posets of regular CW complexes. The first main result of this paper is given in Theorem \ref{thmintervals} where we classify intervals of length $\leq 3$ in diamond transitive posets, in particular showing that these coincide with the corresponding intervals in CW posets.  In the construction of Khovanov homology, there is a certain `sprinkling of signs' on the edges of the Hasse diagram of the Boolean lattice needed in order to obtain a chain comlpex. In Section \ref{bcsec} we introduce \textit{balanced colorings}, the appropriate generalization of this sprinking of signs to thin posets. We define a regular CW complex $X(P)$ whose 1-skeleton is the Hasse diagram $H(P)$, and show in Theorem \ref{dthomology} that if $P$ is a thin poset with $\hat{0}$, $P$ is diamond transitive if and only if $X(P)$ is simply connected. In Proposition \ref{colorable}, we give a sufficient condition for balanced colorability in terms of the space $X(P)$. In Section \ref{CWposets}, we study obstructions to diamond transitivity. In particular, in Theorem \ref{pinchthm}, we show that a poset $P$ is diamond transitive if and only if $P$ does not contain any `pinch products' of thin posets. This result allows us to show in Theorem \ref{cwdt} that CW posets are diamond transitive. In Section \ref{functorsonposets}, we show how to construct a chain complex $C^*(P,\mc{A},F,c)$ from a functor $F:P\to\mc{A}$ on a thin poset $P$ equipped with a balanced coloring $c$. In Theorem \ref{balindept}, we show these complexes are independent of the balanced coloring $c$ up to isomorphism. In Section \ref{categoryofpoco}, we introduce a category with objects $(P,\mc{A},F,c)$ in such a way that passing to cohomology $H^*(P,\mc{A},F,c)$ is functorial, and in Lemma \ref{naturality}, we show that the isomorphisms in Theorem \ref{balindept} are natural. Finally, in Section \ref{seccateg}, we address the purpose of this paper, how to view cohomology theories $H^*(P,\mc{A},F)$ as categorifications of what we will call rank alternators $\RGF(P,R,f)=\sum_{x\in P}(-1)^{\rk(x)}f(x)$ of functions $f:P\to R$ where $R$ is a commutative ring. Theorem \ref{categorify} (stated below) provides a blueprint for how one can construct such categorifications.

\begin{theorem*}
Let $P$ be a balanced colorable thin poset, $R$ a ring, and $f:P\to R$. Suppose $\mc{A}$ is a monoidal abelian category equipped with an isomorphism  of rings $K_0(\mc{A})\cong R$. 
Suppose $F:P\to\mc{A}$ is a functor with the property that for each $x\in P$, $[F(x)]=f(x)$. Then $H(P,\mc{A},F)$ categorifies $\RGF(P,R,f)$ in the sense that $[H(P,\mc{A},F)]=\RGF(P,R,f)$ in $K_0(\mc{A})\cong R$.
\end{theorem*}

Many invariants arise in combinatorics and topology in the form $\RGF(P,R,f)$. One explanation for the ubiquity of such invariants is the following. Let $P$ be an Eulerian poset with $\hat{0}$ and $\hat{1}$. Consider a poset invariant $f$ (that is, a function $f:\{\textnormal{posets}\}\to R$ which depends only on the isomorphism type of the poset). One can think of such invariants equivalently as families of functions $f_P:P\to R$ indexed by posets $P$ by assigning to the element $x$ the corresponding invariant value $f([\hat{0},x])$, and we identify $f(P)$ with $f_P(\hat{1})$. Now consider the poset invariant $g_P$ defined by $g_P(y)=\sum_{x\leq y\in P}f_P(x)$. Then M\"{o}bius inversion guarantees 
\begin{equation}
    f(P)=\sum_{x\in P}(-1)^{\rk(\hat{1})-\rk(x)}g_P(x).
\end{equation}
For example, Euler characteristics of simplicial complexes, determinants, h-polynomials, Tutte polynomials, and Jones polynomials all arise in the form of rank alternators on thin posets. Notice these invariants which arise as rank alternators do not strongly depend on the partial order on $P$ ($\RGF(P,R,f)$ sees only the parity of the rank of each element). In this context, categorification via Theorem \ref{categorify} can be thought of as the process of upgrading these invariants in such a way that the value of the upgraded invariant carries more information about the partial order (in the categorified setting, we have a morphism $F(x)\to F(y)$ for each cover relation $x\lessdot y$). 

The author was surprised to find that this work can be viewed within a more general framework. There is a way to obtain a cohomology theory from a functor on any poset (not just thin posets), by viewing the poset as a topological space. Finite posets are equivalent to finite $T0$-topological spaces, where a poset $(P,\leq)$ corresponds to the topological space with underlying set $P$, whose open sets are those sets $U$ such that $x\in U$ and $y\geq x$ implies $y\in U$ (or in other words, open sets are upper order ideals). Presheaves on the $T0$-space corresponding to a given poset are equivalent to functors on the poset, the functor value at an element $x\in P$ corresponding to the stalk of the corresponding presheaf at $x$. Given a presheaf $F$ on a poset $P$, the global section functor $F\mapsto \lim_{\leftarrow P}F$ is left exact, and thus one can define a cohomology theory $H^*(P,F)=R^*\lim_{\leftarrow P}F$ where $R^i\lim_{\leftarrow P}F$ is the $i$th right derived functor of the global section functor. Thus for thin posets, we have two different ways of obtaining cohomology theories from functors. These theories coincide in some special cases including Boolean lattices. See, for example, the work of Turner and Everitt  \cite{turnereveritthomotopy} where they show that Khovanov homology can be obtained from the right derived functors of the global section functor for the presheaf on the Boolean lattice corresponding to the functor defined via Khovanov's cube construction. More generally, in \cite{turnereverittcellular}, the same authors give a class of posets called \textit{cellular posets}, including both CW posets and geometric lattices, for which these theories coincide.


\section{Thin Posets and CW Posets}
\label{posetbackground}

Here we review some needed terminology in the theory of posets.  Nearly all of the material in Section \ref{posetbackground} is standard and can be found in \cite{stanley1998enumerative}. For the more obscure material, citations will be provided along the way. All posets in this paper are assumed to be finite. 
If $P$ has a unique minimal (respectively maximal) element, we denote this element by $\hat{0}$ (respectively $\hat{1}$). A \textit{lower order ideal} in $P$ is a subset $L$ of $P$ with the property that if $y\in L$ and $x\leq y$ then $x\in L$. A map $f:P\to Q$ between posets is \textit{order preserving} if $x\leq_P y$ implies $f(x)\leq_Q f(y)$ for all $x,y\in P$, and is called an \textit{order embedding} if the converse of this implication also holds. Notice that an order embedding is necessarily injective. An order embedding $f:P\to Q$ is an \textit{isomorphism} if $f$ is also surjective. A \textit{subposet} of $P$ is a poset $P^\prime$ together with an injective order preserving map $f:P^\prime\to P$. A subposet $P^\prime$ is \textit{induced} if $f:P^\prime\to P$ is an order embedding.

A \textit{cover relation} in a poset is a pair $(x,y)\in P\times P$ with $x\leq y$ such that there is no $z\in P$ with $x<z<y$, and in this case we write $x\lessdot y$. Let $C(P)$ denote the set of all cover relations in $P$. The \textit{Hasse diagram} of a finite poset is the directed graph $(P,C(P))$ with vertex set $P$ and a directed edge drawn from left to right from $x$ to $y$ if and only if $x\lessdot y$ (Hasse diagrams are normally drawn bottom up, but in this paper we draw them from left to right as this convention will be more useful visually for the construction of cohomology theories). We will call a poset \textit{connected} if its Hasse diagram is connected. 

A \textit{chain} in a poset $P$ is a totally ordered subposet $C\subseteq P$. We say that a chain $C$ is \textit{from $x$ to $y$} if $x$ (resp. $y$) is the smallest (resp. largest) element of $C$. A chain in $P$ is \textit{saturated} if every cover relation in $C$ is a cover relation in $P$. A chain is $\textit{maximal}$ if it is not properly contained in any other chain. Equivalently, a maximal chain is a saturated chain from a minimal element of $P$ to a maximal element of $P$. 
Given a poset $P$, the \textit{order complex}, $\Delta(P)$, is the abstract simplicial complex whose faces are the chains in $P$. That is, $\Delta(P)=\{F\subseteq P \ | \ F\ \textnormal{is totally ordered}\}$. One should note that we always have $\varnothing\in\Delta(P)$. 

A poset $P$ is \textit{graded} if there is a function $\rk:P\to\N$ with the property that $x\lessdot y$ implies $\rk(y)=\rk(x)+1$. Given a graded poset $P$, the \textit{rank} of $P$, denoted $\rk(P)$, is the maximum length of any chain in $P$, where the length of a chain $C$ is defined as $|C|-1$. If $P$ has $\hat{0}$, we will always assume that $\rk(\hat{0})=0$, and so in this case, we will have $\rk(P)=\max\{\rk(x) \ | \ x\in P\}$. If $P$ is a graded poset, and $x\leq y$ in $P$, the \textit{length} of the closed interval $[x,y]=\{z\in P \ | \ x\leq z\leq y\}$ is $\rk(y)-\rk(x)$. Equivalently, the length of $[x,y]$ is the length of any saturated chain from $x$ to $y$. 

\begin{defn}[{\cite[Section 4]{bjorner1984posets}}]
A  graded poset $P$ is \textit{thin} if every closed nonempty interval of length 2 has exactly 4 elements. In a thin poset, a closed nonempty interval of length 2 is called a \textit{diamond}. 
\label{thindiamond}
\end{defn}

\begin{remark}
There is another class of posets referred to as thin posets in \cite{bonanzinga2011combinatorics}. The above definition of thin posets is taken from \cite{bjorner1984posets} and has nothing to do with the one given in \cite{bonanzinga2011combinatorics}. 
\end{remark}

We now give a class of examples which we will make heavy use of throughout this paper. In every example involving the concept of a `face poset' we will always assume the existence of the `empty face' which is contained in every other face, thus such posets will always have unique minimal elements $\hat{0}$. If $P$ is a poset with $\hat{0}$, we will often write $\bar{P}=P\setminus\{\hat{0}\}$.

\begin{ex}
Let $S$ be a set. The \textit{Boolean lattice}, $2^S$, is the collection of subsets of $S$ partially ordered by inclusion. The Boolean lattice is graded by cardinality, and is thin since any interval of length 2 is of the form $\{T,T\cup\{x\}, T\cup\{y\}, T\cup\{x,y\}\}$. See Figure \ref{thinposetexamples} for an example of the Hasse diagram of a Boolean lattice.
\label{Booleanlatex}
\end{ex}

\begin{ex}
Let $\Delta$ be an \textit{abstract simplicial complex}, that is, a collection of subsets of a finite set $V$ such that $\{x\}\in\Delta$ for each $x\in V$ and for each . The reader is referred to \cite[Section 2.1]{kozlov2007combinatorial} for background information on abstract simplicial complexes and face posets.  The \textit{face poset} of $\Delta$ is the poset $\mc{F}(\Delta)$, with underlying set $\Delta$, partially ordered by inclusion. Every interval in $\mc{F}(\Delta)$ is isomorphic to a Boolean lattice, and therefore face posets of abstract simplicial complexes are also thin posets. Since the face poset of a simplex is itself a Boolean lattice, Example \ref{Booleanlatex} is a special case of this example. Simplicial complexes give an effective way to study topological spaces using combinatorial methods due to the following fact: Given a topological space $X$, suppose there is a simpicial complex $\Delta$ such that $||\Delta||$ is homeomorphic to $X$. In this case, $\Delta$ is referred to as a \textit{triangulation} of $X$. Then the homeomorphism type of $X$ is determined by the poset $\mc{F}(\Delta)$ in the sense that $X\cong \Delta(\mc{F}(\Delta))$. In fact, the order complex $\Delta(\mc{F}(\Delta))$ is the barycentric subdivision of $\Delta$ (see for example \cite[Section 1]{wachs2006poset}).
\label{simpcomplexex}
\end{ex}

\begin{ex}
Let $\Gamma$ be a \textit{polyhedral complex}, that is, a collection $\Gamma$ of polyhedra in $\R^n$ for some $n$ such that (1) if $F\in\Gamma$ and $E$ is a face of $F$, then $E\in\Gamma$, and 
(2) the intersection of any two polyhedra in $\Gamma$ is a face of each.
See \cite[Section 2.2]{kozlov2007combinatorial} for background information on polyhedral complexes. The \textit{face poset} of $\Gamma$ is the poset $\mc{F}(\Gamma)$ with underlying set $\Gamma$, partially ordered by inclusion. Example \ref{simpcomplexex} is a special case of this example since a simplicial complex is simply a polyhedral complex, all of whose polyhedra are simplices. See Figure \ref{thinposetexamples} for an example of the Hasse diagram of a polytopal complex (in this case, just a single polygon). 
\label{polycplxex}
\end{ex}

\begin{ex}
\label{bruhat} 
\label{bruhatsn}
For any Coxeter group $W$, the Bruhat order $\Br(W)$ is a graded poset, with rank function given by the length of a reduced word. See \cite[Section 2]{bjorner2006combinatorics} for background information on the Bruhat order. Bj\"{o}rner noted in \cite[Lemma 2.7.3]{bjorner2006combinatorics} that $\Br(W)$ is a thin poset.
\end{ex}
\begin{ex}
Recall a \textit{ball} in a topological space $X$ is a subspace $B$ of $X$ such that $B$ is homeomorphic to the subspace $\{x\in\R^d \ : \ |x|\leq 1\}$ of $\R^d$ for some $d$. The \textit{interior} of a ball is the image of $\{x\in\R^d \ : \ |x|< 1\}$ under such a homeomorphism.
Let $X$ be a Haussdorff space with a \textit{regular CW decomposition}, $\Gamma$. That is, $\Gamma$ is a collection of closed balls in $X$ such that the interiors of balls in $\Gamma$ partition $X$, and for each ball $\sigma\in\Gamma$ of positive dimension, the boundary $\partial\sigma$ is a union of balls in $\Gamma$. In this case, $\Gamma$ is called a \textit{regular CW complex}, and one will often write $\norm{\Gamma}=\cup_{\sigma\in\Gamma}\sigma$ so in particular $\norm{\Gamma}=X$. The space $\norm{\Gamma}$ is called the \textit{geometric realization} of $\Gamma$. 
There is a more general way to decompose a topological space, using so-called CW complexes \cite[Chapter 0]{Hatcher2002topology}, which are defined by specifying attaching maps to glue balls of different dimensions together. One may alternatively define regular CW complexes as those CW complexes whose attaching maps are all homeomorphisms. 

Given a regular CW complex $\Gamma$, its \textit{face poset}, $\mc{F}(\Gamma)$, is the poset with underlying set $\Gamma$, partially ordered by containment.  See \cite[Appendix A2.5]{bjorner2006combinatorics} or \cite{bjorner1984posets} for background information on regular CW complexes and their face posets. 
It was noted by Bj\"{o}rner that face posets of regular CW complexes are thin \cite[Proposition 3.1 and Figure 1(a)]{bjorner1984posets}. This is not true in general for face posets of CW complexes (for example, consider the CW decomposition of the $2$-disk with one 0-cell, one 1-cell, and one $2$-cell). Following the terminology in \cite{bjorner1984posets}, we will refer to face posets of regular CW complexes as \textit{CW posets}. See Definition \ref{cwposetdefn} for a poset theoretic definition of CW posets and Theorem \ref{cwresults} to see that the two definitions are equivalent. As shown by Bj\"{o}rner in  \cite[Section 2]{bjorner1984posets}, Examples \ref{Booleanlatex}, \ref{simpcomplexex}, \ref{polycplxex} and \ref{bruhatsn} are all CW posets. From the perspective of a topological combinatorialist, regular CW complexes $\Gamma$ are important because the combinatorics of $\mc{F}(\Gamma)$ determines the topology of $\norm{\Gamma}$ in the sense that $\norm{\Gamma}\cong \Delta(\mc{F}(\Gamma))$ \cite[Fact A.2.5.2]{bjorner2006combinatorics}.  This is not true in general for face posets of CW complexes. Thus regular CW complexes can be thought of as those CW complexes which can be understood combinatorially. 
\end{ex}

\begin{defn}[{\cite[Definition 2.1]{bjorner1984posets}}]
A \textit{CW poset} is a poset $P$ with $\hat{0}$, and at least one other element, for which any open interval of the form $(\hat{0},x)$ is homeomorphic to a sphere (that is, the order complex $\Delta(\hat{0},x)$ is homeomorphic to a sphere).
\label{cwposetdefn}
\end{defn}

The following properties of CW posets are proven in \cite{bjorner1984posets} Proposition 2.6, Figure 1(a), and Sections 2.3-2.5.

\begin{theorem}[\cite{bjorner1984posets}]
\label{cwresults}
Let $P$ be a finite poset.
    \begin{enumerate}
        \item $P$ is a CW poset if and only if $P$ is a face poset of a regular CW complex.
        \item If $P$ is a CW poset then $P$ is thin.
        \item If $P$ is a CW poset and $L$ is a lower order ideal in $P$, then $L$ is a CW poset.
        \item If $P$ is a CW poset and $x\in P$, then $\{y\in P \ | \ y\geq x\}$ is a CW poset.
        \item Face posets of simplicial complexes, Bruhat orders on Coxeter groups, and face posets of polytopal complexes are all CW posets.
    \end{enumerate}
\end{theorem}

\begin{ex}
An \textit{Eulerian poset} is a poset in which every nontrivial interval has the same number of elements in even rank as in odd rank. This condition immediately implies that intervals of length two are diamonds, so all Eulerian posets are thin. In fact, Stanley showed in \cite[Section 1]{stanley1994survey}, that CW posets are Eulerian, so this example includes all of the previous examples as special cases.
\end{ex}

\begin{figure}[h]
 	\begin{center}
        \begin{tikzpicture}[scale=1.2, xscale=1.2]
            \useasboundingbox (-2,-1) rectangle (9,1.3);
            \node at (-1.5,1){$2^{[3]}$};
            \node at (2.5,1){$\Br(S_3)$};
            \node[scale=1] at (6,1){$\mc{F}(\usebox\hexagonp)$};
            \node at (0,0){\usebox\booleanlat};
            \node[xscale=.8] at (3.75,0){\usebox\bruhatordersimple};
            \node at (7.5,0){\usebox\polytopefaces};
        \end{tikzpicture}
    \end{center}
    \caption{
    In this figure we see the Boolean lattice $2^{[3]}$ consisting of all subsets of $[3]=\{1,2,3\}$ (or equivalently the face poset of the 2-simplex), the Bruhat order on the symmetric group $S_3$, and the face poset of a hexagon.}
    \label{thinposetexamples}
    \end{figure}
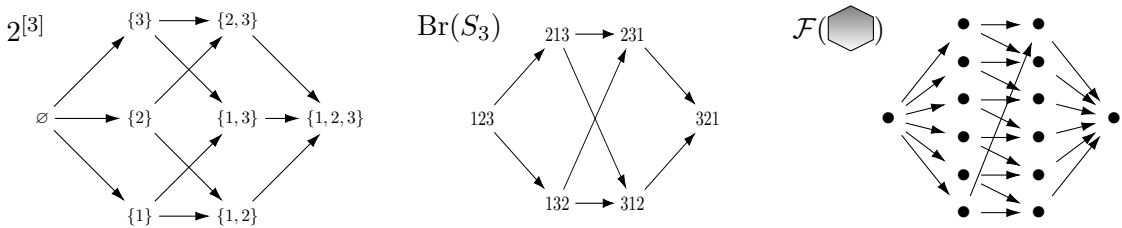

\section{The Diamond Group and Diamond Transitivity}
\label{dtsec}

Given a thin poset $P$, let $S$ denote the set of diamonds in $P$ (see Definition \ref{thindiamond}). Let $W(P)$ denote the group with presentation $\langle S \ | \ R\rangle$, where $R=\{d^2 \ | \ d\in S\}$.  
Given $x,y\in P$ with $x\leq y$, let $\mc{C}_{x,y}$ denote the set of all maximal chains in the interval $[x,y]$ and let $\mc{C}_P$ (or just $\mc{C}$ when $P$ is understood) denote the disjoint union of all of the sets $\mc{C}_{x,y}$ for each pair $(x,y)$ where $x\leq y$ in $P$. In other words, $\mc{C}=\mc{C}_P$ is the set of all saturated chains in $P$. 

\begin{defn}
Given a diamond $d=\{x,a,b,y\}\in S$ (where $x\lessdot a\lessdot y$ and $x\lessdot b\lessdot y$) and a saturated chain $C\in\mc{C}$, define $dC\in \mc{C}$ as follows:
\begin{itemize}
    \item If $x\lessdot a\lessdot y$ is a subchain of $C$, define $dC=(C\setminus\{a\})\cup\{b\}$,
    \item If $x\lessdot b\lessdot y$ is a subchain of $C$, define $dC=(C\setminus\{b\})\cup\{a\}$,
    \item Otherwise, define $dC=C$.
\end{itemize}
We will refer to the process of passing from $C$ to $dC$ as performing a \textit{diamond move} on $C$.
\end{defn}

\begin{remark}
    The above action of diamonds on maximal chains coincides with Stanley's $\tau_i$ operators in \cite[Section 6]{stanley2009promotion}, defined en route to generalizing the operations of promotion and evacuation to the context of graded posets. 
\end{remark}

 Intuitively, if $C$ contains one side of the diamond $d$, we form $dC$ by rerouting $C$ along the other side of the diamond $d$. 
 Notice that for any $d\in S$, and any $C\in\mc{C}$, $d(dC)=C$, and therefore this defines an action of $W(P)$ on $\mc{C}$ via $(d_1\dots d_{k-1}d_k)(C)=d_1(\dots d_{k-1}(d_k(C))\dots)$.
 Let $N(P)$ denote the subgroup of $W(P)$ generated by all words which act trivially on $\mc{C}$. The following two facts are special cases of a more general well known result, but we provide proofs here for convenience.

\begin{lemma}
For any thin poset $P$ with diamond set $S$, $N(P)$ is a normal subgroup of $W(P)$. 
\end{lemma}
\begin{proof}
Let $w\in N(P), y\in W(P)$, and $C\in\mc{C}$. Then $$y^{-1}wy(C) =y^{-1}(w(y(C)) =y^{-1}(y(C)) =(yy^{-1})(C) =eC =C.$$ Therefore $y^{-1}wy\in N$ and we conclude that $N(P)$ is normal.
\end{proof}
\begin{defn}
    Given a thin poset $P$, the \textit{diamond group}, $D(P)$ is defined as the quotient of $W(P)$ by the normal subgroup $N(P)$:
    $$D(P)=W(P)/N(P).$$
\end{defn}
Recall that a group action of a group $G$ on a set $X$ can be thought of as a homomorphism $\phi:G\to S_X$ from $G$ to the symmetric group on the set $X$, where the permutation $\phi(g)$ is defined via $\phi(g)(x)=gx$ for $g\in G, x\in X$. A group action is \textit{faithful} if the corresponding homomorphism $\phi_G:G\to S_X$ is injective, or in other words, if the only element of $G$ which fixes every element of $X$ is the identity element.

\begin{lemma}
Given a thin poset $P$, $D(P)$ acts faithfully on the set $\mc{C}$ of saturated chains. In other words, the homomorphism $\phi: D(P)\to S_{\mc{C}}$ is injective.
\label{faithful}
\end{lemma}
\begin{proof}
Given $w\in W(P)$ let us denote the class of $w$ in $D(P)$ by $[w]$. The action of $W(P)$ on $\mc{C}$ decends to an action of $D(P)$ on $\mc{C}$ via $[w]C=wC$. This is well defined since if $[w]=[v]$, then $w^{-1}v\in N(P)$ and thus $w^{-1}v$ acts trivially on all elements of $\mc{C}$. Therefore, for all $C\in\mc{C}$, $wC=w(w^{-1}v)(C)=vC$. This is a group action since $[v]([w]C)=v(wC)=(vw)(C)=([vw])C=([v][w])C$ for all $v,w\in W(P)$ and $C\in\mc{C}$.  Suppose $[w]\in D(P)$ satisfies $[w]C=C$ for all $C\in\mc{C}$. Then $wC=C$ for any $C\in\mc{C}$ and therefore $w\in N(P)$ and therefore $[w]$ is the identity element in $D(P)$. 
\end{proof}

Recall that if a group $G$ acts on a set $X$, and $Y\subseteq X$ is a subset which is closed under the action of $G$, then $G$ also acts on $Y$ by restriction. In our case, notice that for each $x\leq y$, $\mc{C}_{x,y}\subseteq \mc{C}$ is closed under the action of $D(P)$. An action of $G$ on $X$ is called \textit{transitive} if for any $x,y\in X$ there exists $g\in G$ such that $gx=y$. In other words, if we let $\mc{O}_x=\{gx \ | \ g\in G\}$ denote the \textit{orbit} of $x$ under the action of $G$, then $G$ acts transitively on $X$ if $X=\mc{O}_x$ for some (equivalently for all) $x\in X$.

\begin{defn} A thin poset $P$ is \textit{diamond transitive} if $D(P)$ acts transitively on $\mc{C}_{x,y}$ for each pair $x\leq y$ in $P$.
\end{defn}

\begin{ex}
Let $P=\Br(S_3)=\usebox\symthree$\ ,  $$S=\left\{\usebox\symthreeDone,\usebox\symthreeDtwo, \usebox\symthreeDthree, \usebox\symthreeDfour\right\}$$ and  $$\mc{C}_{\hat{0},\hat{1}}=\left\{\usebox\symthreeCone,\usebox\symthreeCtwo, \usebox\symthreeCthree, \usebox\symthreeCfour\right\}.$$ 
Then $D(P)$ acts on $\mc{C}_{\hat{0},\hat{1}}$ according to the table in Figure \ref{diamondmult}. 
\end{ex}

\begin{figure}
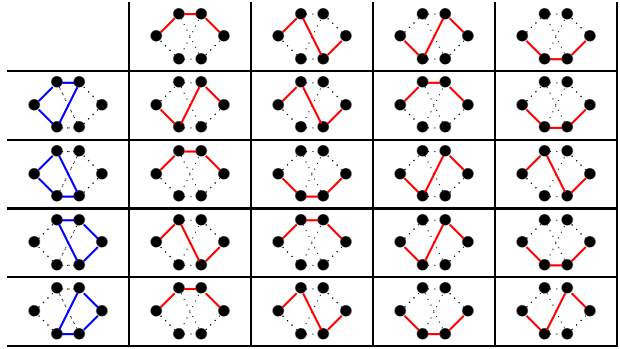

\centering
\begin{tabular}{l|l|l|l|l|}
 & \usebox\symthreeCone & \usebox\symthreeCtwo & \usebox\symthreeCthree & \usebox\symthreeCfour \\ \hline
\usebox\symthreeDone & \usebox\symthreeCthree & \usebox\symthreeCtwo & \usebox\symthreeCone & \usebox\symthreeCfour \\ \hline
\usebox\symthreeDtwo & \usebox\symthreeCone & \usebox\symthreeCfour & \usebox\symthreeCthree & \usebox\symthreeCtwo \\ \hline
\usebox\symthreeDthree & \usebox\symthreeCtwo & \usebox\symthreeCone & \usebox\symthreeCthree & \usebox\symthreeCfour \\ \hline
\usebox\symthreeDfour & \usebox\symthreeCone & \usebox\symthreeCtwo & \usebox\symthreeCfour & \usebox\symthreeCthree \\ \hline
\end{tabular}
\caption{The top row shows the set of maximal chains (in red), and the leftmost column shows the set of diamonds (in blue) for the Bruhat order on $S_3$. The maximal chain shown in row $d$ and column $C$ is the chain $dC$ defined by the action of $D(P)$ on $\mc{C}_{\hat{0},\hat{1}}$.}
\label{diamondmult}
\end{figure}

We now give a characterization of small intervals in diamond transitive thin posets. 
    
\begin{theorem}\label{thmintervals}
    Intervals of length 2 and 3 in a diamond transitive thin poset $P$ are of the combinatorial type prescribed in Figure \ref{intervals}.
\end{theorem}
\begin{proof}
    The result for intervals of length 2 follows from thinness. Consider now an interval $[x,y]$ of length 3 in a diamond transitive thin poset. Take a chain of the form $x\lessdot a_0\lessdot b_0\lessdot y$ in $[x,y]$. The interval $[x,b_0]$ must be a diamond, so there must be $a_1\neq a_0\in P$  with $x\lessdot a_1\lessdot b_0$. The interval $[a_1,y]$ must be a diamond, so there must be an element $b_1\neq b_0\in P$ with $a_1\lessdot b_1\lessdot y$. The interval $[a_0,y]$ must be a diamond so either $a_0\lessdot b_1$ or we continue the process by adding elements $a_2,b_2$ with $a_2\lessdot b_1$ and $a_2\lessdot b_2$. Since $P$ is finite, one must eventually stop adding such elements $a_i,b_i$. Now, the interval $[x,b_0]$ must be a diamond, so eventually we must connect $a_0$ to $b_k$ for some $k$. 

    Once we reach $k$ such that $a_0\lessdot b_k$, we claim there can be no more elements in $[x,y]$ besides those already considered.
    Suppose there were some element $a_{k+1}$ in $[x,y]$ with $x\lessdot a_{k+1}$. If $a_{k+1}\lessdot b_i$ for some $i\leq k$ then thinness is contradicted so there must be some $b_{k+1}$ with $a_{k+1}\lessdot b_{k+1}\lessdot y$. If $a_i\lessdot b_{k+1}$ for any $i\leq k$ then again thinness is contradicted. However now, there is no way to get to the chain $x\lessdot a_{k+1}\lessdot b_{k+1}\lessdot y$ from the chain $x\lessdot a_k\lessdot b_k\lessdot y$ by diamond moves, thus contradicting diamond transitivity. Thus every interval of length 3 is of the combinatorial type shown in the statement above.
\end{proof}

    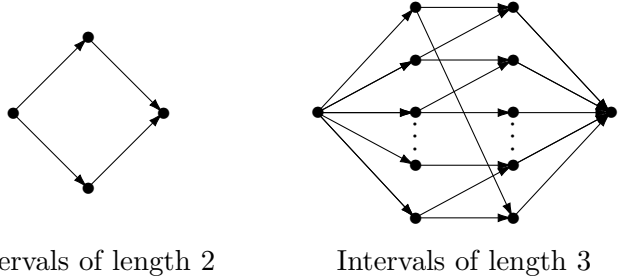
\begin{figure}
        \centering
            \begin{tikzpicture}
            \useasboundingbox (0,-2) rectangle (5,1.5);
            \node[rotate=-90] at (0,0){\usebox\diamonddd};
            \node at (0,-2){Intervals of length 2};
            \node[rotate=-90] at (5,0){\usebox\lengththree};
            \node at (5,-2){Intervals of length 3};
            \end{tikzpicture}
        \caption{The combinatorial type of intervals of length 2 and 3 in diamond transitive posets.}
        \label{intervals}
    \end{figure}

\begin{remark}\label{cwremark}
Every CW poset has intervals of length 2 and 3 isomorphic to those shown in Theorem \ref{thmintervals}, as noted by Bj\"{o}rner in \cite[Figure 1]{bjorner1984posets}. Thus one may wonder the relationship between diamond transitive thin posets and CW posets. Theorem \ref{cwdt} together with Example \ref{dtnotcw} show that the CW posets are a proper subset of the diamond transitive posets. 
\end{remark}

The following theorem, while less general than the forthcoming Theorem \ref{cwdt}, provides an opportunity to present a direct proof of diamond transitivity. 

\begin{theorem}
Face posets of simplicial complexes are diamond transitive.
\end{theorem}

\begin{proof} Let $P$ be the face poset of a simplicial complex. Recall intervals in face posets of simplicial complexes are Boolean lattices. Consider an interval $[x,y]$ in $P$. Then by the above remark, $[x,y]\cong 2^{[n]}$ for some $n$. The set of elements in $2^{[n]}$ of a given rank (that is, all subsets of $[n]$ of a given size) can be totally ordered lexicographically, and this induces a lexicographic on maximal chains. Let $C_0$ denote the lexicographically largest chain $\varnothing\subseteq\{n\}\subseteq\{n-1,n\}\subseteq\dots\subseteq [n]$. 

We now show, by induction on the distance $d(C,C_0)$ in lexicographic order between $C$ and $C_0$ that there exists $w\in D(P)$ such that $wC=C_0$. The base case is trivial since for distance 0 we have $C=C_0$. Otherwise, since $C$ is not lexicographically largest, there is a subchain $S\subseteq S+i\subseteq S+i+j$ of $C$ with $i<j$. Now consider the (lexicographically larger) chain $dC$ where $d=\{S,S+i,S+j,S+i+j\}$. By induction we have $w\in D(P)$ such that $w(dC)=C_0=(wd)C$.  
\end{proof}

\section{The Diamond Space and Balanced Colorability}
\label{bcsec}

In this section, we give a topological interpretation of diamond transitivity. This interpretation turns out to be useful in regard to understanding certain $\{1,-1\}$ colorings (called balanced colorings) of the cover relations of $P$, which we use in Section \ref{functorsonposets} to construct chain complexes associated to functors on thin posets.

\begin{defn}
Given a thin poset $P$, we construct a regular CW complex $X(P)$ as follows:
\begin{enumerate}
    \item The 0-skeleton, $X^0(P)$ of $X(P)$ is the set $P$.
    \item The 1-skeleton, $X^1(P)$ of $X(P)$ is the (undirected) Hasse diagram of $P$, that is, we glue in a 1-cell connecting pairs $(x,y)$ whenever $x\lessdot y$.
    \item The 2-skeleton $X^2(P)=X(P)$ of $X(P)$ is gotten by gluing 2-cells into each diamond in $P$. Since we will add no higher dimensional cells, $X(P)$ is equal to the 2-skeleton. 
\end{enumerate}
Call the resulting regular CW complex $X(P)$ the \textit{diamond space} of $P$. 
\end{defn}


\begin{ex}\label{kcrownex}
Let $P$ be a diamond transitive thin poset, and $I$ an interval in $P$ of length 3. Then by Theorem \ref{thmintervals}, $I$ is isomorphic to the poset shown in Figure \ref{intervals} on 2n vertices for some $n$. If we disregard the long diagonal from rank 1 to rank 2 in Figure \ref{intervals}, we see that gluing in all diamonds produces a 2-dimensional disk $D_1$. Adding back in the long diagonal produces two more diamonds, giving another 2-dimensional disk $D_2$ the diamond space $X(I)$ is exactly $D_1$ glued to $D_2$ along their common boundary, and therefore $X(I)$ is a 2-sphere.
\end{ex}

Let $P$ be a thin poset. Fix a continuous map $\phi:X^1(P)\to\R$, so that for each $a\in X^0(P)$, $\phi(a)=\rk(a)$, and on each closed 1-cell $e$ in $X^1(P)$, corresponding to an edge in $H(P)$ from $a$ to $b$, $\phi|_e$ is a linear isomorphism onto the interval $[\rk(a),\rk(b)]\subseteq \R$. 

\begin{defn}
Let $r_1,r_2\in\R$. A continuous map $\gamma:[r_1,r_2]\to X^1(P)$ is \textit{monotonically increasing} (resp. \textit{monotonically decreasing}) if for all $s,t\in[r_1,r_2]$, $s<t$ implies $\phi(\gamma(s))<\phi(\gamma(t))$ (resp. $\phi(\gamma(s))>\phi(\gamma(t))$).
\end{defn}

\begin{lemma}\label{monotonic}
Let $P$ be a diamond transitive thin poset with $\hat{0}$. Any continuous map $\gamma:[0,1]\to X(P)$ with $\gamma(0)=\gamma(1)=\hat{0}$ is homotopic to a piecewise linear map $\gamma^\prime:[0,1]\to X^1(P)$ with $\gamma^\prime(0)=\gamma^\prime(1)=\hat{0}$ such that there exists $r\in(0,1)$ so that $\gamma^\prime|_{[0,r]}$ is monotonically increasing and $\gamma^\prime|_{[r,1]}$ is monotonically decreasing.
\end{lemma}

\begin{proof}
Let $\gamma:[0,1]\to X(P)$. Using the cellular approximation theorem \cite[Theorem 4.8]{Hatcher2002topology}, we see that $\gamma$ is homotopic to a map $\bar{\gamma}$ whose image is contained in $X^1(P)$. Then, by applying the simplicial approximation theorem \cite[Theorem 2C.1]{Hatcher2002topology} to $\bar{\gamma}$, we see that $\gamma$ is homotopic to a map $\tilde{\gamma}$ that is simplicial with respect to some iterated barycentric subdivision of $[0,1]$. In other words, there is a partition $0=r_0<r_1<\dots<r_{2k}=1$ of $[0,1]$ such that 
\begin{enumerate}
    \item $\tilde{\gamma}(r_0)=\tilde{\gamma}(r_{2k})=\hat{0}$
    \item $\tilde{\gamma}(r_i)\in X^0(P)$ for each $i$ 
    \item For $0\leq i\leq 2k-1$, $\tilde{\gamma}(r_i)\lessdot \tilde{\gamma}(r_{i+1})$ or $\tilde{\gamma}(r_{i+1})\lessdot \tilde{\gamma}(r_{i})$
    \item If $r\in[0,1]\setminus\{r_0,\dots,r_{2k}\}$ then $\tilde{\gamma}(r)\in X^1(P)\setminus X^0(P)$
    \item For $0\leq i\leq 2k-1$, $\tilde{\gamma}$ is linear (in particular either monotonically increasing or decreasing) on $[r_i,r_{i+1}]$. 
\end{enumerate}
Choose $i_1<\dots<i_{2\ell}$ where $\{i_1,\dots,i_{2\ell}\}\subseteq\{1,\dots,2k\}$ such that $\tilde{\gamma}$ is monotonically increasing on $[0,r_{i_1}]$, monotonically decreasing on $[r_{i_1},r_{i_2}]$, monotonically increasing on $[r_{i_1},r_{i_3}]$ and so on. We will show by induction on $\ell$ (that is, half the number of subintervals in such a partition) that $\tilde{\gamma}$ is homotopic to a map $\gamma^\prime$ as desired in the statement of this lemma. If $\ell=1$, then $\tilde{\gamma}$ monotonically increases on $[0,r_{i_1}]$, and monotonically decreases on $[r_{i_1},1]$, so we are done. If $\ell>1$ we proceed as follows. Recall that paths $\alpha,\beta$ can be multiplied to form the product path $\alpha\beta$ which traverses first $\alpha$ and then $\beta$. Also recall that, given a path $\alpha:[a,b]\to X$ in a topological space $X$, we let $\alpha^{-1}$ denote the path defined by $\alpha^{-1}:[-b,-a]\to X$ where $\alpha^{-1}(t)=\alpha(-t)$, so $\alpha^{-1}$ traverses the same points as $\alpha$ but in the opposite direction (see \cite[Section 1.1]{Hatcher2002topology}). In the Hasse diagram $H(P)$, $\tilde{\gamma}|_{[0,r_{i_1}]}$ corresponds to a directed path from $\hat{0}$ to some element $x_1\in P$. Now, $(\tilde{\gamma}|_{[r_{i_1},r_{i_2}]})^{-1}$ corresponds to a directed path in $H(P)$ starting at some element $x_2$ and ending at $x_1$. Let $\hat{0}=y_0\lessdot y_1\lessdot\dots\lessdot y_n=x_2$ be a directed path in $H(P)$, and let $\beta:[0,1]\to X^1(P)$ be a monotonically increasing piecewise linear map corresponding to this directed path. Then $\beta (\tilde{\gamma}|_{[r_{i_1},r_{i_2}]})^{-1}$ is piecewise linear and monotonically increasing from $\hat{0}$ to $x_1$. Since $P$ is diamond transitive, there is a sequence of diamond moves taking the path $\tilde{\gamma}|_{[0,r_{i_1}]}$ to the path $\beta(\tilde{\gamma}|_{[r_{i_1},r_{i_2}]})^{-1}$. Each diamond move defines a homotopy in $X(P)$ taking a path along one side of the diamond to the path along the other side. Thus, $\tilde{\gamma}=\tilde{\gamma}|_{[0,r_{i_1}]}\tilde{\gamma}|_{[r_{i_1},1]}$ is homotopic to the path $\beta (\tilde{\gamma}|_{[r_{i_1},r_{i_2}]})^{-1}\tilde{\gamma}|_{[r_{i_1},1]}$ which is homotopic to  $\gamma^\prime=\beta (\tilde{\gamma}|_{[r_{i_2},1]})$. Notice that $\gamma^\prime$ has a partition of its domain $0<r_{i_3}<\dots<r_{i_{2\ell}}$ with $2(\ell-1)$ subintervals so that $\gamma^\prime$ is monotonic on each subinterval. Thus we are done by induction.
\end{proof}

\begin{theorem}
Let $P$ be a thin poset with $\hat{0}$. Then $P$ is diamond transitive if and only if $\pi_1(X(P),\hat{0})$ is trivial.
\label{dthomology}
\end{theorem}
\begin{proof}
Suppose that $P$ is diamond transitive. By Lemma \ref{monotonic}, any loop $\gamma:[0,1]\to X(P)$ based at $\hat{0}$ can be homotoped to a loop $\gamma^\prime:[0,1]\to X^1(P)$ such that there exists $r\in(0,1)$ for which $\gamma^\prime|_{[0,r]}$ increases monotonically to some element $x\in P$ and $\gamma^\prime|_{[r,1]}$ decreases monotonically from $x$ back to $\hat{0}$. Let $C_1=(\hat{0}=x_0\lessdot x_1\lessdot\dots\lessdot x_k=x)$ be the saturated chain in $P$ corresponding to $\gamma^\prime|_{[0,r]}$ and let $C_2=(\hat{0}=y_0\lessdot y_1\lessdot\dots\lessdot y_k=x)$ be the saturated chain in $P$ corresponding to $(\gamma^\prime|_{[r,1]})^{-1}$. Since $P$ is diamond transitive, there is a sequence of diamond moves that takes $C_1$ to $C_2$, and this defines a homotopy from $\gamma^\prime|_{[0,r]}$ to $(\gamma^\prime|_{[r,1]})^{-1}$. Thus $\gamma^\prime$ is homotopic to $\gamma^\prime|_{[0,r]}(\gamma^\prime|_{[0,r]})^{-1}$ which is homotopic to the constant map at $\hat{0}$. Thus we have shown that any loop in $X(P)$ based at $\hat{0}$ is nullhomotopic. 

Conversely, let $\hat{0}=x_0\lessdot\dots\lessdot x_k=x$ and $\hat{0}=y_0\lessdot\dots\lessdot y_k=x$ be saturated chains.
Correspondingly, we have paths $\beta,\gamma:[0,1]\to X(P)$ proceeding along the edges of each chain, and by assumption $\beta\gamma^{-1}$ is nullhomotopic, or equivalently, there is a homotopy $h:[0,1]\times[0,1]\to X(P)$ so that $h(t,0)=\beta(t)$ and $h(t,1)=\gamma(t)$ for all $t\in[0,1]$, $h(0,s)=\hat{0}$, and $h(1,s)=x$ for all $s\in[0,1]$. 
Applying cellular approximation to this homotopy, we get a subdivision of $[0,1]\times[0,1]$ into sub-rectangles, each sub-rectangle corresponding to a diamond move in $P$.
\end{proof}

\begin{defn}[{\cite[Section 2.7]{bjorner2006combinatorics}}]
A \textit{balanced coloring} on a thin poset $P$ is a function $c:C(P)\to\{1,-1\}$ if each diamond in $P$ is assigned an odd number of $-1$ colored edges. A thin poset $P$ is \textit{balanced colorable} if it admits a balanced coloring.
\label{balanceddef}
\end{defn}

We have seen in Theorem \ref{dthomology} a relationship between the combinatorial property of diamond transitivity in $P$ and the topology of the space $X(P)$. Next, we show that the property of balanced colorability can also be thought of in terms of the diamond space $X(P)$. Balanced colorings of $P$ are $\{1,-1\}$ colorings of the edges in the Hasse diagram $H(P)$. Thus balanced colorings can be viewed as elements of the cochain group $C^1(X(P),\Z_2)=\Hom(C_1(X(P),\Z_2)$. Warning: the coefficient groups for cohomology are usually written with additive notation, however because of our situation, we will use multiplicative notation, identifying $(\{0,1\},+)$ with $(\{1,-1\},\cdot)$. 
Consider the function $\phi_0:C_2(X(P))\to\Z_2=\{1,-1\}$ sending every 2-cell to $-1\in\Z_2$. Recall, 2-cells correspond to diamonds in $P$, so we will denote 2-cells by $d$ and each 2-cell has boundary consisting of four edges $\{e_1^d,e_2^d,e_3^d,e_4^d\}$. Given $\psi:C_1(X(P))\to \Z_2$, $\delta\psi:C_2(X(P))\to\Z_2$ is defined on 2-cells $d$ with boundary $\{e_1^d,e_2^d,e_3^d,e_4^d\}$ by $\delta\psi(d)=\psi(e^d_1)\psi(e^d_2)\psi(e^d_3)\psi(e^d_4)$. If $\delta\psi=\phi_0$, then $\psi(e^d_1)\psi(e^d_2)\psi(e^d_3)\psi(e^d_4)=-1$ for each diamond $d$, or in other words, $\psi$ is a balanced coloring. Since $X(P)$ has no 3-cells, $\phi_0$ is guaranteed to be a cocycle. Since we are working over the field $\Z_2$, we have $H^k(X(P),\Z_2)\cong \Hom(H_k(X(P),\Z_2),\Z_2)$ for all $k$. Thus we have proven the following.

\begin{prop}\label{weakcolorable}
Let $\phi_0:C_2(X(P))\to\Z_2=\{1,-1\}$ denote the cocycle sending every 2-cell to $-1\in\Z_2$. Then $P$ is balanced colorable if and only if $[\phi_0]$ is trivial in $H^2(X(P),\Z_2)$. 
\end{prop}

In particular, this says that if $H^2(X(P),\Z_2)$ is trivial, then $P$ is balanced colorable. However, after working out any number of simple examples, the reader will notice that this is very rarely going to be the case. The following result will be applicable in a wider variety of situations.

\begin{prop}\label{colorable}
Let $P$ be a thin poset and suppose that $X(P)$ forms the 2-skeleton of a regular CW complex $Z(P)$ in which each 3-cell has an even number of 2-dimensional faces, and suppose that $H^2(Z(P),\Z_2)=0$. Then $P$ is balanced colorable.
\end{prop}
\begin{proof}
Since $X(P)$ and $Z(P)$ have the same 2-skeleton, we have that $C^2(X(P),\Z_2)=C^2(Z(P),\Z_2)$ so we can identify $\phi_0$ as a 2-cochain in $Z(P)$. Let $F$ be a 3-cell in $Z(P)$. By assumption, $F$ has an even number of 2-dimensional faces, $d_1,\dots, d_{2n}$. In our multiplicative notation ($\Z_2=\{1,-1\}$) we have $\delta\phi_0(F)=\prod_{i=1}^{2n}\phi_0(d_i)=(-1)^{2n}=1$. Therefore $\delta\phi_0$ is the identity element in $C^3(Z(P),\Z_2)$, so $\phi_0\in C^2(Z(P),\Z_2)$ is a cocycle. Since $H^2(Z(P),\Z_2)=0$, $\phi_0$ is a coboundary. Thus $\phi_0=\delta\psi$ for some $\psi:C_1(Z(P))\to\Z_2$. Since $X(P)$ and $Z(P)$ have the same 1-skeleton, we have $C_1(X(P))=C_1(Z(P))$ so we can identify $\psi:C_1(X(P))\to\Z_2$ as a 1-cochain in $X(P)$. Therefore we have that $[\phi_0]$ is trivial in $H^2(X(P),\Z_2)$ and so by Propsitition \ref{weakcolorable}, $P$ is balanced colorable.
\end{proof}

\begin{remark}
Given a poset $P$, in Theorem \ref{colorable} we are asking whether there exists a regular CW complex with certain properties whose 1-skeleton is the Hasse diagram of $P$. This is a generalization of the following question asked by Hersh \cite{hersh2018posets}: which posets $P$ arise as 1-skeleta of convex polytopes? This question arose in quite a different context, as a condition under which one can prove the nonrevisiting path conjecture.
\end{remark}

\begin{ex}
As an easy application of Proposition \ref{colorable}, we see that the Boolean lattice $B_n$ is balanced colorable since $X(B_n)$ is the 2-skeleton of the $n$-cube. Note that every 3-cell in an $n$-cube is a 3-cube, and thus has an even number of faces. 
Khovanov makes use of such a balanced coloring in  \cite[Section 3.3]{khovanov1999categorification}, en route to constructing a homology theory which categorifies the Jones polynomial.
\end{ex}

\begin{ex}
Let $P$ be a diamond transitive thin poset. Then by Example \ref{kcrownex}, one can construct a space $Z_3(P)$ by gluing in 3-cells to $X(P)$ along each interval of length 3 in $P$, and it is clear from the picture in Figure \ref{intervals} that each such 3-cell has an even number of 2-dimensional faces. One might expect that since we have ``filled'' all of the obvious 2-dimensional ``holes'' in $X(P)$, that we must have $H^2(Z(P),\Z_2)=0$. However, one can check that for the poset in Example \ref{dtnotcw} that this is not the case.

With this in mind, one might consider the following construction. Given a poset $P$, let $\Int(P)$ denote the poset of intervals of $P$ partially ordered by containment. Given that $P$ is diamond transitive, one might hope that $\Int(P)$ is the face poset of a regular CW complex $Z(P)$ with the properties required in Proposition \ref{colorable}. In fact, in the case that $P$ is the face poset of the boundary of a polytope, Bj\"{o}rner shows in \cite[Theorem 6.1]{bjorner1984posets} that $\Int(P)$ is also the face poset of regular CW decomposition of a sphere. This, together with Proposition \ref{colorable} gives an alternative proof that face posets of polytopes are balanced colorable. This suggests the following question: for which thin posets $P$ does it follow that $\Int(P)$ is the face poset of a regular CW complex $Z(P)$ whose second cohomology vanishes and all of whose 3-cells have an even number of 2-dimensional faces?
\end{ex}

We remain unaware whether all diamond transitive posets are balanced colorable, but we consider ourselves content, for now, to have the following result of Bj\"{o}rner.

\begin{theorem}[{\cite[Corollary 2.7.14]{bjorner2006combinatorics}}]
    CW posets are balanced colorable.
    \label{cwbalanced}
\end{theorem}

\begin{proof}
    This follows from \cite[Corollary 2.7.14]{bjorner2006combinatorics} by noticing the proof is not specific to the Bruhat order, but holds for all CW posets.
\end{proof}

In Section \ref{categoryofpoco}, we consider a category structure on tuples $(P,\mc{A},F,c)$ where $c$ is a balanced coloring. Morphisms in this category thus require a way to pass between different balanced colorings. 

\begin{defn}
Given a thin poset $P$, a \textit{central coloring} of $P$ is a map $c:C(P)\to\{1,-1\}$ such that every diamond has an even number of $-1$'s. Let $\Phi=\Phi_P$ denote the group of all edge colorings $f:C(P)\to\Z_2=\{1,-1\}$ of the Hasse diagram, with group operation defined by pointwise multiplication of functions. Let $\Phi^b\subseteq \Phi$ denote the set of balanced colorings and let $\Phi^c\subseteq \Phi$ denote the set of central colorings. Let $1=1_\Phi$ denote the identity element of $\Phi$, that is, the constant function on $C(P)$ with value 1.
\end{defn}

\begin{lemma}
    \label{phic}
    For any $c_1,c_2\in\Phi^b$, the product $c_1c_2$ is in $\Phi^c$.
\end{lemma}
\begin{proof}
    Consider a diamond with edges $e,f,g,h$. Then \[c_1(e)c_1(f)c_1(g)c_1(h)=-1 \ \ \ \ \ \ \textnormal{and} \ \ \ \ \ \ c_2(e)c_2(f)c_2(g)c_2(h)=-1\] so \[(c_1c_2)(e)(c_1c_2)(f)(c_1c_2)(g)(c_1c_2)(h)=1. \qedhere \] 
\end{proof}

\begin{lemma}\label{phicaction}
The subgroup $\Phi^c\leq\Phi$ acts freely and transitively on $\Phi^b$.
\end{lemma}
\begin{proof}
    Given $c\in \Phi^c$ and $d\in\Phi^b$, since $c$ has product 1 over any diamond, and $d$ has product -1 over any diamond, $cd$ has product -1 over any diamond, so $cd\in\Phi^b$.
    Notice that for any $c\in\Phi$, $c^2=1_{\Phi}$. Let $d=c_1c_2$ (an element of $\Phi^c$ by Lemma \ref{phic}). Then $c_1d=c_1c_1c_2=1_\Phi c_2=c_2$.
\end{proof}

\begin{remark}
    In \cite[Section 5]{marietti2007}, Marietti constructs operators $\Phi_S$ which play a role analogous to that played by central colorings in this paper, and a weaker version of Lemma \ref{phicaction} appears in his setting as \cite[Corollary 6.3]{marietti2007}.
\end{remark}

\begin{corollary}
    Let $P$ be a balanced colorable thin poset with balanced coloring $c$. Then multiplication by $c$ gives a bijection between $\Phi^b$ and $\Phi^c$. In particular, $|\Phi^b|=|\Phi^c|$.
\end{corollary}
\begin{proof}
The map is well defined by Lemma \ref{phic}, and the inverse (also multiplication by $c$) is well defined by Lemma \ref{phicaction}.
\end{proof}



\begin{lemma}
If $P$ is a diamond transitive thin poset with $\hat{0}$, then for any $c\in \Phi^c$ there exists $f:P\to\{1,-1\}$ such that $\delta f=c$. 
\label{centralimage}
\end{lemma}
\begin{proof}
Since $P$ is diamond transitive with $\hat{0}$, it follows from Theorem \ref{dthomology} that we have $H^1(X(P),\Z_2)=0$. Since $c\in\Phi^c$, we have $c\in\ker\delta^1=\im\delta^0$. 
\end{proof}

In fact we can do a bit better than the previous Lemma and actually assign an explicit function $f$ so that $\delta f=c$. It turns out that the so-called `greedy algorithm' (Definition \ref{greedydefn}) does the trick. 

\begin{defn}
Let $P$ be a thin poset, and $c\in\Phi^c$. A \textit{partial coloring} of $P$ is a function $f:S_f\to\{1,-1\}$ where $S_f\subseteq P$.
Say that an edge $z\lessdot y$ in $S_f$ is \textit{coherent} with respect to the pair $(f,c)$ if $c(z\lessdot y)=f(z)f(y)$. Say that a saturated chain from $\hat{0}$ to an element $x$ in $P$ is \textit{coherent} with respect to $(f,c)$ if each cover relation on the chain is coherent with respect to $(f,c)$. Call a partial coloring $f$ of $P$ \textit{allowable} (with respect to $c$) each cover relation in $S_f$ is coherent with respect to $(f,c)$.
\end{defn}

\begin{defn}\label{greedydefn}
Fix a linear ordering on each rank of the poset $P$. We will call this collection of linear orderings $\mc{O}$. Given $c\in\Phi^c_P$ we define a function $\gr_\mc{O}^c:P\to\{1,-1\}$ as follows (we will often simply write $\gr^c$ and omit the $\mc{O}$ but keep in mind that the function could depend on $\mc{O}$).  
\begin{enumerate}
    \item Assign all elements of rank $i=0$ with the value $1$. 
    \item Once all elements of rank $i$ are colored, we proceed to color elements of rank $i+1$ as follows. Once the first $j$ elements (with respect to $\mc{O}$) of rank $i+1$ are assigned colors, the $(j+1)$st element is assigned the color $1$ if this produces an allowable partial coloring and otherwise we assign the color $-1$. 
\end{enumerate}
\end{defn}

To show that the above process gives a well defined function $\gr_\mc{O}^c$ with $\delta(\gr_\mc{O}^c)=c$, it remains only to show that in the case that assigning the color $1$ to the $(j+1)$st element $u$ of the $i$th rank does not assign an allowable partial coloring, that assigning the value $-1$ does give an allowable partial coloring. Let us denote by $f$ this partial coloring, assigning the value $-1$ to $u$, and we assume inductively that if we erase the color from $u$, we have an allowable partial coloring. The following lemma will prove useful in this regard.

\begin{lemma}
Under the inductive assumption of the previous paragraph, if $d$ is a diamond in $S_f$, and if $C$ is a saturated chain in $S_f$ from $\hat{0}$ to  $u$, coherent with respect to $(f,c)$, then $dC$ is coherent with respect to $(f,c)$. 
\label{coherence}
\end{lemma}
\begin{proof}
If the top element of $d$ has rank $\leq i$, then all edges of $d$ are coherent with respect to $(f,c)$ by definition, so the only nontrivial case is when $u$ is the top element of $d$. In this case, we have a diamond of the form shown in Figure \ref{diamondgreedy}. 
    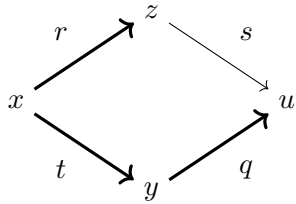
\begin{figure}
        \centering
        \begin{tikzpicture}
            \node at (0,0){\usebox\diamondgreedy};
        \end{tikzpicture}
        \label{diamondgreedy}
        \caption{The only case of interest in the proof of Lemma \ref{coherence} to show that diamond moves preserve coherence of saturated chains. The element $u$ is at rank $i+1$ and the thickened edges are coherent by assumption.}
    \end{figure}
The three edges shown as thickened are coherent with respect to $(f,c)$ by assumption, so it suffices to show that the fourth edge is also coherent with respect to $(f,c)$. But since we have $t=f(x)f(y)$, $q=f(y)f(u)$, $r=f(x)f(z)$, and $rstq=1$, it follows that $s=rtq=f(z)f(u)$.
\end{proof}

Since assigning the value $1$ to $u$ does not yield an allowable partial coloring, there is some cover relation $y\lessdot u$ for which $-f(y)= c(y\lessdot u)$, or in other words, $(y\lessdot u)$ is coherent with respect to $(f,c)$. To show that $f$ is an allowable partial coloring with respect to $c$, we must show that any other edge of the form $z\lessdot u$ is coherent with respect to $(f,c)$. Consider any element $z$ in rank $i$, such that there is a cover relation $z\lessdot u$. In the case that $P$ is a diamond transitive thin poset with $\hat{0}$, there exist saturated chains $C_1$ and $C_2$ containing the edges $y\lessdot u$ and $z\lessdot u$ respectively, both with minimum $\hat{0}$ and maximum $u$. By diamond transitivity there exist a sequence of diamond moves connecting $C_1$ and $C_2$, and by assumption, $C_1$ is coherent with respect to $(f,c)$. Using Lemma \ref{coherence} iteratively, we obtain the following.

\begin{prop}\label{greedyprop}
Let $P$ be a diamond transitive thin poset with $\hat{0}$, and $c\in\Phi^c$. The function $\gr_\mc{O}^c$ satisfies $\delta(\gr_\mc{O}^d)=c$.
\end{prop}

Not only do we now have a canonical preimage of a given central coloring under the differential $\delta$, but the map $\gr^c$ has some very useful properties.

\begin{lemma}
The map $\gr^c_\mc{O}$ depends only on $c$ and the choice $\gr^c(\hat{0})=1$, and is therefore independent of the ordering $\mc{O}$.
\label{independent}
\end{lemma}
\begin{proof}
We prove this by induction on rank. At rank 0, we have $\gr^c(\hat{0})=1$. Suppose $\rk(y)>0$. Then by construction, $\gr^c(y)=\gr^c(x)c(x\lessdot y)$ for any $x\lessdot y$. By our inductive assumption, $\gr^c(x)$ depends only on $c$ and the choice $\gr^c(\hat{0})=1$, as desired. 
\end{proof}

\begin{remark}
We will henceforth omit any mention of an ordering $\mc{O}$. From Lemma \ref{independent}, we can infer that there are two possible preimages of a given central coloring $c$ under the differential $\delta$, one given by the choice $\gr^c(\hat{0})=1$ and the other, which we might denote $\bar{\gr}^c$ determined by the choice $\bar{\gr}^c(\hat{0})=-1$. These are related of course by $\bar{\gr}^c=-{\gr}^c$.
\end{remark}

\begin{corollary}
    \label{uniqueness}
    Suppose $P$ is a diamond transitive thin poset with $\hat{0}$ and let $f:P\to\{1,-1\}$. If $f$ satisfies $f(\hat{0})=1$ and $\delta(f)=c$, then $f=\gr^c.$
\end{corollary}

\begin{prop}
    \label{greedyprod}
    Let $P$ be a diamond transitive thin poset with $\hat{0}$. Given $b,b^\prime\in\Phi_P^c$, we have $\gr^{bb^\prime}=\gr^b\gr^{b^\prime}.$
    \end{prop}
    
\begin{proof}
    By uniqueness (that is, Corollary \ref{uniqueness}), it suffices to show that $\delta(\gr^b\gr^{b^\prime})=bb^\prime$. Given a cover relation $x\lessdot y$ in $P$, we have 
        \begin{align*}
            \delta(\gr^b\gr^{b^\prime})(x\lessdot y)
            &=(\gr^b\gr^{b^\prime})(x)(\gr^b\gr^{b^\prime})(y)\\
            &=\gr^b(x)\gr^b(y)\gr^{b^\prime}(x)\gr^{b^\prime}(y)\\
            &=\delta(\gr^b)(x\lessdot y)\delta(\gr^{b^\prime})(x\lessdot y)\\
            &=b(x\lessdot y) b^\prime(x\lessdot y)\\
            &=bb^\prime(x\lessdot y). \qedhere
        \end{align*}
\end{proof}

The following proposition involves concepts (Definition \ref{CPOE} and the discussion before Lemma \ref{colorops}) from Section \ref{categoryofpoco}, however the result feels more at home in this section.

\begin{prop}
    \label{greedyinv}
    Let $P,Q$ be a diamond transitive thin posets with $\hat{0}$, let $c$ be a central coloring of $Q$, and suppose $\phi:P\to Q$ is a cover preserving order embedding. Then
    \[\gr^{\phi^{-1}(c)}=\gr^c\circ\phi.\]
\end{prop}
\begin{proof}
    Again by uniqueness, the following calculation suffices:
        \begin{align*}
            \delta(\gr^c\circ\phi)(x\lessdot y)
            &=\gr^c(\phi(x))\gr^c(\phi(y))\\
            &=c(\phi(x)\lessdot\phi(y))\\
            &=\phi^{-1}(c)(x\lessdot y).
        \end{align*}
    We also need to remark that $\gr^c\circ\phi(\hat{0}_P)=\gr^c(\hat{0}_Q)=1$, so by Corollary \ref{uniqueness} we get the desired result.
\end{proof}

\section{Obstructions to Diamond Transitivity}
\label{CWposets}

In this section, we begin by constructing a family of thin posets which are not diamond transitive. We then show that this construction, in a sense, provides the only possible type of obstruction to diamond transitivity. This gives us a classification of diamond transitive posets in terms of obstructions, and this classification allows us to prove that all CW posets are diamond transitive.

\begin{defn}
Given thin posets $P$ (with $\hat{0}_P$ and $\hat{1}_P$) and $Q$ (with $\hat{0}_Q$ and $\hat{1}_Q$) of the same rank $n\geq 3$, define the \textit{pinch product} $P\pinch Q$ by
$$P\pinch Q=(P-\{\hat{0}_P,\hat{1}_P\})\amalg (Q-\{\hat{0}_Q,\hat{1}_Q\})\amalg\{\hat{0},\hat{1}\}$$
where $\amalg$ denotes disjoint union. A partial order on $P\pinch Q$ is defined by setting $x\leq y$ if any of the following conditions hold:
\begin{enumerate}
    \item $x=\hat{0}$ or $y=\hat{1}$,
    \item $x,y\in P-\{\hat{0}_P,\hat{1}_P\}$ and $x\leq_P y$,
    \item $x,y\in Q-\{\hat{0}_Q,\hat{1}_Q\}$ and $x\leq_Q y$.
\end{enumerate}
  \end{defn}
 In words, the pinch product is formed by removing the $\hat{0}$ and $\hat{1}$ from each poset, taking the disjoint union, and then adding a new $\hat{0}$ and $\hat{1}$ to the result (or in other words, by identifying $\hat{0}_P\sim\hat{0}_Q$ and $\hat{1}_P\sim\hat{1}_Q$). 
 \begin{lemma}
 Given thin posets $P$ and $Q$ of the same rank $n\geq 3$, $P\pinch Q$ is thin.
 \end{lemma}
 \begin{proof}
 The pinch product $P\pinch Q$ is graded with rank function
 \begin{equation*}
   \rk(x)=%
   \begin{cases}
     0 &\text{if}\  x=\hat{0}, \\
     n &\text{if} \ x=\hat{1},\\
     \rk_P(x)&\text{if} \ x\in P-\{\hat{0}_P,\hat{1}_P\},\\
     \rk_Q(x)&\text{if} \ x\in Q-\{\hat{0}_Q,\hat{1}_Q\}.\\
   \end{cases}
\end{equation*}
Any interval of length two can be identified as either an interval of length two in $P$ or an interval of length two in $Q$ and is thus a diamond.
 \end{proof}
 
\begin{ex}
Let $P=\Br(S_3)\pinch \Br(S_3)$. The Hasse diagram of $P$ is shown below: 
    
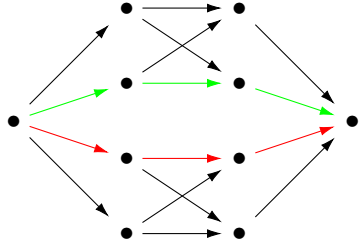
\begin{center}
        \begin{tikzpicture}
        \node[rotate=-90] at (0,0){\usebox\thincounter};
        \end{tikzpicture}
    \end{center}
Notice the chains colored green and red are in different orbits of $D(P)$, so $P$ is not diamond transitive. In fact, it is not hard to see that $P$ is the smallest possible non-diamond transitive thin poset. This poset also appears in \cite[Figure 2]{stanley1994survey} as an example of an Eulerian poset which is not a CW poset. It is not hard to see that $P$ is also the smallest possible thin poset with $\hat{0}$ which is not diamond transitive (it is the only such example among posets with 10 or less elements). 
\label{pinchex}
\end{ex}

\begin{lemma}
    Given thin posets $P,Q,$ with $\hat{0}$ and $\hat{1}$, of equal rank $n\geq 3$, there is a group isomorphism  $f: D(P)\times D(Q)\to D(P\pinch Q)$ such that the following diagram commutes:
        \begin{center}
            \begin{tikzcd}
                D(P)\times D(Q)\arrow[hookrightarrow, d]\arrow[r,"f"]
                &D(P\pinch Q)\arrow[hookrightarrow, d]\\
                S_{\mc{C}_P}\times S_{\mc{C}_Q}\arrow[hookrightarrow, r]
                &S_{\mc{C}}
            \end{tikzcd}
        \end{center}
    \label{comdiag}
\end{lemma}
\begin{proof}
Let $S$ (resp. $S_P$, $S_Q$) be the set of diamonds in $P\pinch Q$ (resp. $P,Q$), and let $\mc{C}$ (resp. $\mc{C}_P$, $\mc{C}_Q$) be the set of saturated chains in $P\pinch Q$ (resp. $P,Q$). Then $S=S_P\amalg S_Q$ and $\mc{C}=\mc{C}_P\amalg\mc{C}_Q$. Furthermore notice that if $d\in S_P$ and $e\in S_Q$ that $de=ed$ in $D(P\pinch Q)$. Therefore we can factor $D(P\pinch Q)\cong D(P)\times D(Q)$ by writing any word $d_1\dots d_k\in D(P)$ in the form $(e_1\dots e_r)(f_{r+1}\dots f_k)$ where $e_i\in S_P$ and $f_j\in S_Q$ for all $i,j$. We define the map $f$ by sending $(e_1\dots e_r, f_1\dots f_s)\in D(P)\times D(Q)$ to $e_1\dots e_r f_1\dots f_s\in D(P\pinch Q)$. Commutativity of the above diagram follows from the group action axioms.
\end{proof}

\begin{lemma}
For any two thin posets $P,Q,$ with $\hat{0}$ and $\hat{1}$, of equal rank $n\geq 3$, $P\pinch Q$ is not diamond transitive. 
\end{lemma}
\begin{proof}
Given a maximal chain $C$ in $P$, a maximal chain $C^\prime$ in $Q$, and $w\in D(P\pinch Q)$, the action of $f^{-1}(w)$ on $C$ gives another maximal chain in $P$, so by the commutative diagram in Lemma \ref{comdiag}, $wC\neq C^\prime$ for any $w\in D(P\pinch Q)$. 
\end{proof}

We wish to show that a thin poset $P$ is diamond transitive if and only if there are no pinch products ``living in" $P$. We now describe the appropriate notion of ``living in".

\begin{defn}
An induced subposet $S$ of $P$ is \textit{saturated} if every cover relation in $S$ is also a cover relation $P$.
A saturated subposet $S$ of a thin poset $P$ is \textit{diamond-complete} if for every diamond $d$ in $P$, and every saturated chain $C$ in $S$, $dC\subseteq S$. 
\end{defn}

\begin{remark}\label{dcinterval}
Every interval of a thin poset $P$ is clearly a diamond-complete thin subposet of $P$, but it is possible to have diamond complete subposets which are not intervals, for example, either copy of $\Br(S_3)$ in $\Br(S_3)\pinch\Br(S_3)$.
\end{remark}

\begin{defn}
Given a saturated chain $C$ in a thin poset $P$,
let $\mc{O}_C$ denote the orbit of $C$ with respect to the action of $D(P)$. 
Let $C$ be a saturated chain in a thin poset $P$. Define $P_C$ to be the induced subposet of $P$ with underlying set $\cup_{C^\prime\in\mc{O}_C}C^\prime$.
\end{defn}


\begin{lemma}
Let $P$ be a thin poset, and let $C$ be a saturated chain in $P$. Then $P_C$ is a diamond-complete thin subposet of $P$. 
\label{saturated}
\end{lemma}
\begin{proof}
Every saturated chain in $P_C$ is a saturated chain in $P$ by construction. Given a diamond $d$ in $P$ and a saturated chain $C^\prime$ in $P_C$, $dC^\prime\in\mc{O}_C$ and therefore $dC^\prime\subseteq P_C$ by definition, so $P_C$ is diamond complete. To see that $P_C$ is thin, suppose that $x<y$ in $P_C$ and $\rk(y)=\rk(x)+2$. Since $x<y$ is not saturated in $P$, it must not be saturated in $P_C$. Therefore there is an element $a\in P_C$ with $x<a<y$. Extend this to a maximal chain $C^\prime$ in $P_C$. Since $P$ is thin, there exists another element $b\in P$ with $x<b<y$, and consider the diamond $d=\{x,a,b,y\}$. By definition of $P_C$, all elements of $dC^\prime$ are in $P_C$ and therefore $b\in P_C$, so the interval $[x,y]$ in $P_C$ consists of four elements $\{x,a,b,y\}$. 
\end{proof}

\begin{lemma}
Given $x<y$ in $P$ with $\rk(y)\geq \rk(x)+3$, suppose $C$ and $D$ are saturated chains from $x$ to $y$ such that $\mc{O}_C\neq\mc{O}_D$ and $P_C\cap P_D=\{x,y\}$ . Then $P_C\cup P_D$ is a diamond-complete thin subposet of $P$ isomorphic to $P_C\pinch P_D$. 
\label{inducedpinch}
\end{lemma}

\begin{proof}
First we show that $P_C\cup P_D$ is induced by induction on the length $\ell$ of $C$. The base case is for $\ell=3$. In this case, $C=x\lessdot c_1\lessdot c_2\lessdot y$ and $D=x\lessdot d_1\lessdot d_2\lessdot y$.  Suppose towards a contradiction that  $P_C\cup P_D$ is not induced. Then without loss of generality we must have $c_1\lessdot d_2$. However in this case, the open interval $(c_1,y)$ consists of at least 3 elements, contradicting thinness of $P$. Suppose now that $\ell>3$, and again suppose towards a contradiction that  $P_C\cup P_D$ is not induced. In this case, $C=x\lessdot c_1\lessdot \dots\lessdot c_{\ell-1}\lessdot y$ and $D=x\lessdot d_1\lessdot \dots\lessdot d_{\ell-1}\lessdot y$. Then without loss of generality there is a cover relation $c_i\lessdot d_{i+1}$ for some $1\leq i\leq \ell-2$. Let $C^\prime$ denote the chain $x\lessdot c_1\lessdot \dots\lessdot c_i\lessdot d_{i+1}$ and let $D^\prime$ denote the chain $x\lessdot d_1\lessdot \dots\lessdot d_i\lessdot d_{i+1}$. Then $P_{C^\prime}$ is an induced subposet of $P_C$ and $P_{D^\prime}$ is an induced subposet of $P_D$. Let $C^{\prime\prime}$ denote the chain $c_i\lessdot c_{i+1}\lessdot \dots\lessdot \lessdot y$ and let $D^{\prime\prime}$ denote the chain $c_i\lessdot d_{i+1}\lessdot \dots\lessdot y$. Then $P_{C^{\prime\prime}}$ is an induced subposet of $P_C$ and $P_{D^{\prime\prime}}$ is an induced subposet of $P_D$. By induction, $P_{C^\prime}\cup P_{D^\prime}$ and $P_{C^{\prime\prime}}\cup P_{D^{\prime\prime}}$ are induced and meaning that there are no other cover relations between $C$ and $D$ besides $c_i\lessdot d_{i+1}$ and possibly $d_i\lessdot c_{i+1}$. This contradicts thinness of $P$ since then we have $(c_i,d_{i+2})=\{d_{i+1}\}$. 

We have that $P_C\cup P_D$ is saturated since any saturated chain in $P_C\cup P_D$ is either a saturated chain in $P_C$ or a saturated chain in $P_D$, and so by Lemma \ref{saturated}, is saturated in $P$. Furthermore, $P_C\cup P_D$ is diamond-complete since any saturated chain in $P_C\cup P_D$ lies entirely in $P_C$ or $P_D$ respectively, and thus by Lemma \ref{inducedpinch}, any diamond move on such a chain gives another chain in $P_C$ or $P_D$ respectively. Lastly, $P_C\cup P_D$ is thin since any interval of length two in $P_C\cup P_D$ is either an interval of length two in $P_C$ or in $P_D$ and is therefore a diamond again by Lemma \ref{saturated}.
\end{proof}

\begin{theorem}
Let $P$ be a thin poset. Then $P$ is diamond transitive if and only if $P$ contains no diamond-complete subposet isomorphic to a pinch product of two thin posets. 
\label{pinchthm}
\end{theorem}
\begin{proof}
If $P$ is not diamond transitive, one can find saturated chains $C$ and $D$ from $x$ to $y$ in $P$ for some $x,y\in P$ such that $\mc{O}_C\neq\mc{O}_D$ and $P_C\cap P_D=\{x,y\}$. Lemma \ref{inducedpinch} then provides a diamond-complete thin subposet isomorphic to a pinch product. Conversely, suppose $P$ contains a diamond-complete subposet isomorphic to a pinch product, say of the form $Q=Q_1\pinch Q_2$ with some unique minimal element $\hat{0}_Q$ and maximal element $\hat{1}_Q$. Let $C_1\subseteq Q_1$ and $C_2\subseteq Q_2$ be a saturated chains from $\hat{0}_Q$ to $\hat{1}_Q$ in $P$. Note that $Q_1$ and $Q_2$ are diamond-complete by Lemma \ref{saturated} since $Q_1=P_{C_1}$ and $Q_2=P_{C_2}$. For any sequence of diamonds $d_1\dots d_r$ in $P$, $(d_1\dots d_r)C_1\subseteq Q_1$, so $C_1$ and $C_2$ are in different orbits of the action of $D(P)$ and therefore $P$ is not diamond transitive.
%
%
%
\end{proof}

\begin{remark}
With Remark \ref{dcinterval} in mind, one might wish to restate the above theorem as ``a thin poset $P$ is diamond transitive if and only if every diamond-complete subposet of $P$ is an interval of $P$''.
\end{remark}

In Remark \ref{cwremark} and Example \ref{pinchex}, we noted that CW posets share certain similarities with diamond transitive thin posets. We now further explore this connection.

\begin{prop}
If $P$ can be written in the form $P=P_1\pinch P_2$, where $P_1$ and $P_2$ are thin posets of equal rank, then $P$ is not a CW poset. \label{pinchnotcw}
\end{prop}
\begin{proof}
Since $P$ has $\hat{1}$, if $P$ is a face poset of a CW complex $Y$, then  $Y$ must be connected. Suppose to the contrary that $P=P_1\pinch P_2$ is the face poset of a regular CW complex. Removing the top element, $P=P_1\pinch P_2\setminus \hat{1}$ is also the face poset of a regular CW complex $X=X_1\amalg X_2$ where $P_1\setminus \hat{1}$ is the face poset of $X_1$ and $P_2\setminus \hat{1}$ is the face poset of $X_2$. But for $n>1$, there is no way to glue an $n$-cell to an $n-1$ dimensional skeleton which is not connected, and obtain a connected space (i.e. the image of a continuous map $f:e^n\to X_1\amalg X_2$ must be contained in either $X_1$ or $X_2$).  
\end{proof}

\begin{theorem}\label{cwdt}
CW posets are diamond transitive.
\end{theorem}

\begin{proof}
Let $P$ be a CW poset. By Theorem \ref{pinchthm}, it suffices to show that $P$ contains no diamond-complete subposets isomorphic to pinch products. We proceed by induction on $|P|$, with the base case $|P|=2$ obvious. For $|P|>2$, suppose towards a contradiction that $P$ contains a diamond-complete subposet $L$ such that $L$ is isomorphic to a pinch product. If $L=P$ then this contradicts Proposition \ref{pinchnotcw} so we can assume $L$ is a proper subset of $P$.  By Theorem \ref{cwresults} parts 3 and 4, we see that intervals in CW posets are CW posets. Since $L$ is isomorphic to a pinch product, $L$ has $\hat{0}_L$ and $\hat{1}_L$. If $\hat{0}_L\neq\hat{0}_P$, or $\hat{1}_L$ is not maximal in $P$, or if $P$ has no $\hat{1}$, then the interval $[\hat{0}_L,\hat{1}_L]$ in $P$ is a proper subset of $P$, and hence a CW poset. By induction, we can assume that $[\hat{0}_L,\hat{1}_L]$ is diamond transitive. But $[\hat{0}_L,\hat{1}_L]$ contains $L$ as a diamond complete subposet, which contradicts Theorem \ref{pinchthm}. Thus for the remainder of the proof, we can assume that $\hat{0}_L=\hat{0}_P$ and $\hat{1}_L=\hat{1}_P$.

Let $L=L_1\pinch L_2$. Let $x\in P\setminus L$ and $y\in L\setminus\{\hat{0},\hat{1}\}$. If $x\leq y$ then there must be some $x^\prime\in P\setminus L$ and $y^\prime\in L\setminus\{\hat{0},\hat{1}\}$ such that $x^\prime\lessdot y^\prime$. The lower interval $[\hat{0},y^\prime]$ is a CW poset and is therefore diamond transitive by induction. Therefore, any saturated chain from $\hat{0}$ to $y^\prime$ in $L$ can be taken to a saturated chain from $\hat{0}$ which goes through $x^\prime$ to $y^\prime$, via the action of $D(P)$. Thus we have $x^\prime\in L$ by diamond-completeness, giving a contradiction to the assumption that $x^\prime\in P\setminus L$. Similarly we can show that it is impossible to have $y\leq x$. Thus, setting $L^\prime=P\setminus L\cup\{\hat{0},\hat{1}\}$, we conclude that $L^\prime$ is a thin poset with the same rank as $L$, and $L\cap L^\prime=\{\hat{0},\hat{1}\}$, therefore we have $P=L\pinch L^\prime$, contradicting Proposition \ref{pinchnotcw}.
\end{proof}

\begin{remark}
Since CW posets are a special case of Eulerian posets, one might ask whether all Eulerian posets are diamond transitive. However, the poset in Example \ref{pinchex} provides an example of a non-diamond transitive Eulerian poset. 
\end{remark}

\begin{ex}\label{dtnotcw}
This example was suggested to the author by Anton Mellit \cite{mellit}. Let $P$ be the face poset of a triangulation of a 2-torus, with a unique maximal element adjoined. Then $P$ is diamond transitive, but is not a CW poset since $\Delta(\hat{0},\hat{1})$ is a barycentric subdivision of the torus, which by the classification of surfaces is not homeomorphic to a sphere. In fact, since $\mu(\hat{0},\hat{1})=\tilde{\chi}(\Delta(\hat{0},\hat{1}))=-1$ is the reduced Euler characteristic of the torus by Philip Hall's theorem, and for each facet $F$ of the torus, we have $\mu(\hat{0},F)=-1$ (the reduced Euler characteristic of a circle), $P$ is not even Eulerian (the M\"{o}bius function values alternate on ranks for Eulerian posets). Thus we have the following inclusions 
\[
    \{\textnormal{CW posets}\}
    \subset \{\textnormal{diamond transitive posets}\}
    \subset\{\textnormal{thin posets}\}, 
\]
\[
    \{\textnormal{CW posets}\}
    \subset\{\textnormal{Eulerian posets}\}
    \subset\{\textnormal{thin posets}\},
\]
but there exist examples of diamond transitive posets which are not Eulerian, and examples of Eulerian posets which are not diamond transitive (Example \ref{pinchex}). 
\end{ex}

\begin{remark}
The author knows of no diamond transitive Eulerian poset which is not a CW poset. One might ask whether the following is true:
\[\{\textnormal{diamond transitive posets}\}\cap \{\textnormal{Eulerian posets}\}= \{\textnormal{CW posets}\}.\] 
This would give an almost too-good-to-be-true classification of CW posets, so it seems unlikely, but worth asking.
\end{remark}



\section{Cohomology of Functors on Thin Posets}
\label{functorsonposets}

A poset $P$ can be identified as a category whose objects are elements of $P$ and with a unique morphism, denoted $x\leq y$, from $x$ to $y$ if and only if $x\leq y$ in $P$ (otherwise $\Hom(x,y)=\varnothing$). By transitivity of $P$, we can define composition via $(y\leq z)\circ (x\leq y)=x\leq z$. Reflexivity of $P$ guarantees identity morphisms, and antisymmetry says that the only isomorphisms are the identity morphisms. A functor on a poset then can be identified as a labeling of the nodes and edges of the Hasse diagram so that for any $x\leq y$, compositions of morphisms along any two saturated chains between $x$ and $y$ agree. 

Let $P$ be a poset (regarded as a category), let $\{F_x\}_{x\in P}$ be a labeling of the vertices of the Hasse diagram by objects of a category $\mc{C}$, and let $\{F_{y,x}\}_{x\lessdot y\in C(P)}$ be a labeling of the edges $C(P)$ of the Hasse diagram by morphims in $\mc{C}$, where $F_{y,x}:F_x\to F_y$. Suppose that for any $x\leq y$, and any two directed paths $x=x_0,x_1,\dots,x_n=y$, $x=x_0^\prime,x_1^\prime,\dots,x_k^\prime=y$, from $x$ to $y$, we have
$$F_{x_{n},x_{n-1}}\circ\dots\circ F_{x_1,x_0}=F_{x^\prime_{k},x^\prime_{k-1}}\circ\dots\circ F_{x^\prime_1,x^\prime_0}.$$
Then the following defines a functor $F:P\to\mc{C}$:
\begin{equation}
    F(x)=F_x \ \ \ \ \text{for all}\ x\in P
    \label{functone}
\end{equation}
\begin{equation}
    F(x\leq y)=F_{x_{n}, x_{n-1}}\circ\dots\circ F_{x_1, x_0}
    \label{functtwo}
\end{equation}
where $x=x_0\lessdot x_1\lessdot\dots\lessdot x_{n-1}\lessdot x_n=y$ is any directed path in the Hasse diagram from $x$ to $y$. 

\begin{theorem}
Let $P$ be a diamond transitive thin poset, let $\{F_x\}_{x\in P}$ be a labeling of $P$ by objects of a category $\mc{C}$, and let $\{F_{y,x}\}_{x\lessdot y\in C(P)}$ be a labeling of $C(P)$ by morphims in $\mc{C}$, where $F_{y,x}:F_x\to F_y$. Then equations \ref{functone} and \ref{functtwo} determine a functor if and only if this labeling commutes on every diamond in $P$.
\label{dtfunctor}
\end{theorem}


\begin{proof}
Let $C$ and $C^\prime$ denote two saturated chains from $x$ to $y$ where $x\leq y$ in $P$. We proceed by induction on the minimal number of diamond moves needed to take $C$ to $C^\prime$. For the base case, if no diamond moves are needed, then $C=C^\prime$ and we are done. If a positive number of diamond moves are needed, take a minimal word $w=w_1\dots w_k\in D(P)$ such that $wC=C^\prime$. Set $C^{\prime\prime}=w_kC$ and note $(w_1\dots w_{k-1})C^{\prime\prime}=C$. Compositions along $C$ and along $w_kC=C^{\prime\prime}$ agree since we have commutativity on diamonds. The minimum number of diamond moves needed to take $C^{\prime\prime}$ to $C^\prime$ is less than $k$ so by induction compositions along $C^{\prime\prime}$ and $C^\prime$ agree.
\end{proof}

We now define the cochain complex associated to a functor on a thin poset, which Lemma \ref{welldef} shows is well defined.

\begin{defn}\label{thinposethomologydef}
Let $\mc{A}$ be an abelian category, $P$ a finite thin poset with balanced coloring $c$, and $F:P\to\mc{A}$ a covariant functor. Define a cochain complex, denoted $C^*(P,\mc{A},F,c)$ with differential $\delta$ given by the formulas
\begin{equation}
    C^k(P,\mc{A},F,c)=\bigoplus_{\rk(x)=k}F(x)\ \ \ \ \ \ \ \ \ \delta^k=\mathop{\sum_{x\lessdot y }}_{\rk(x)=k} c(x\lessdot y)F(x\lessdot y).
    \label{thincohomdef}
\end{equation}
Denote the cohomology of $C^*(P,\mc{A},F,c)$ by $H^*(P,\mc{A},F,c)$. 
\end{defn}

\begin{remark}
In Definition \ref{thinposethomologydef}, morphisms from $C^k$ to $C^{k+1}$ can be identified as matrices with columns indexed by elements in $P$ of rank $k$ and rows indexes by elements in $P$ of rank $k+1$, where the matrix element indexed by $x$ and $y$ is a morphism $F(x)\to F(y)$. Then $F(x\lessdot y)$ can be identified as the matrix with all 0's except for the morphism $F(x\lessdot y)$ in the matrix element indexed by $x$ and $y$ and $\delta^k$ is the (signed) sum of such matrices.
\end{remark}

\begin{remark}
\label{notationalrmk}
Theorem \ref{balindept} shows that for diamond transitive posets with $\hat{0}$, this cohomology does not depend on the balanced colorings, so we may instead write $H^*(P,\mc{A},F)$. In fact, since the functor $F$ contains the data of $P$ and $\mc{A}$, we may choose to simply write $C^*(F,c)$ and $H^*(F)$. However, due to the categorical structure introduced in Section \ref{categoryofpoco}, we will often use the full data $(P,\mc{A},F,c)$ in our notation.
\end{remark}

\begin{lemma} Let $P$ be a thin poset with balanced coloring $c$, and let $F:P\to\mc{A}$ be a functor, where $\mc{A}$ is an abelian category. Then $C^*(P,\mc{A},F,c)$ is a cochain complex.  \label{welldef} 
\end{lemma}
\begin{proof} Suppose $F$ is covariant. Given $A\in F(x)$, with $\rk(x)=k$, consider the component $\delta^2_{x,z}(A)$ of $\delta^2(A)$ in $F(z)$ where $\rk(z)=\rk(x)+2$. If  $\delta^2_{x,z}(A)$ were nonzero, then by construction we must have $z\geq x$, in which case the interval $[x,z]=\{x,a,b,z\}$ is a diamond. Thus, we have 
\begin{align*}
\delta^2_{x,z}&=c(a\lessdot z)F(a\lessdot z)\circ c(x\lessdot a)F(x\lessdot a)
+c(b\lessdot z)F(b\lessdot z)\circ c(x\lessdot b)F(x\lessdot b)\\
&=c(a\lessdot z)c(x\lessdot a)\big[F(a\lessdot z)\circ F(x\lessdot a)\big]
+c(b\lessdot z)c(x\lessdot b)\big[F(b\lessdot z)\circ F(x\lessdot b)\big]\\
&=[c(a\lessdot z)c(x\lessdot a)+c(b\lessdot z)c(x\lessdot b)]\big[F(a\lessdot z)\circ F(x\lessdot a)\big]\\
&=0.\qedhere
\end{align*}
\end{proof}

\begin{theorem}
Let $P$ be a diamond transitive thin poset with $\hat{0}$ and with balanced colorings $c_1,c_2\in\Phi^b$. Let $\mc{A}$ be an abelian category and $F:P\to\mc{A}$ a functor. Then $A\mapsto \gr^{c_1c_2}(x)A$, where $A\in F(x)$ defines a (natural) isomorphism of complexes $C^*(P,\mc{A},F,c_1)\to C^*(P,\mc{A},F,c_2)$.
\label{balindept}
\end{theorem}
\begin{proof}
Consider the cochain complex $(C^*(X(P),Z_2),\delta)$. Notice that $\ker\delta^1=\Phi^c$ under the identification of $\{1,-1\}$ with $\Z_2=\{0,1\}$. If $P$ is diamond transitive, then $H_1(X(P),\Z_2)=0$ by Proposition \ref{dthomology} and thus $H^1(X(P),\Z_2)=\Hom(H_1(X(P),\Z_2),\Z_2)=0$. Therefore any $c\in\Phi^c$ is a coboundary. Let $c_1,c_2\in\Phi^b$. By Lemma \ref{phicaction} there is some $d\in\Phi^c$ such that $c_1d=c_2$ (in particular, $d=c_1c_2$). Using Proposition \ref{greedyprop}, we have a canonical function $\gr^d:P\to\Z_2$ such that $d=\delta \gr^d$. Note: under the identification of $\Z_2$ with $\{1,-1\}$ this means that $d(x\lessdot y)=\gr^d(x)\gr^d(y)$. Define a chain map from $C^*(P,\mc{A},F,c_1)\to C^*(P,\mc{A},F,c_2)$ with components  $\gr^d(x)1_{F(x)}:F(x)\to F(x)$ for each $x\in P$. In order to check this is a chain map, we must have, for any $x\lessdot y$, $c_2(x\lessdot y)F(x\lessdot y)\gr^d(x)1_{F(x)}=\gr^d(y)1_{F(y)}c_1(x\lessdot y)F(x\lessdot y)$, or equivalently, $c_2(x\lessdot y)\gr^d(x)=\gr^d(y)c_1(x\lessdot y)$. But $c_2=dc_1$, so 
\begin{align*}
    c_2(x\lessdot y)\gr^d(x)&=d(x\lessdot y)c_1(x\lessdot y)\gr^d(x)\\
    &=\gr^d(x)\gr^d(y)c_1(x\lessdot y)\gr^d(x)\\
    &=\gr^d(y)c_1(x\lessdot y). 
\end{align*}
Naturality follows from Lemma \ref{naturality}.
\end{proof}
%

Suppose $F:P\to\mc{A}$ is a covariant functor from a diamond transitive thin poset $P$ to an abelian category $\mc{A}$. Suppose that $P$  has a balanced coloring $c:C(P)\to\{1,-1\}$. If $I$ is an upper or lower order ideal, then $I$ is a thin poset, the restriction $c|_{C(I)}$ is a balanced coloring, and the restriction $F|_I:I\to\mc{A}$ is a functor. If $I$ is an upper order ideal, then $C^*(I,\mc{A},F|_I,c|_I)$ is a subcomplex of $C^*(P,\mc{A},F,c)$, or equivalently,  the inclusion $C^*(I,\mc{A},F|_I,c|_I)\hookrightarrow C^*(P,\mc{A},F,c)$ is a chain map.  


\begin{theorem}
Suppose $P$ is a diamond transitive balanced colorable thin poset, and $F:P\to\mc{A}$ is a functor to an abelian category $\mc{A}$. Given an upper order ideal $I$ in $P$, there is a long exact sequence
\[\dots H^*(I,\mc{A},F)\to H^*(P,\mc{A},F)\to H^*(P\setminus I,\mc{A},F)\xrightarrow{d} H^*(I,\mc{A},F)\dots\] 
where the map $d$ has degree 1.
\label{pocoLES}
\end{theorem}
\begin{proof}
Let $c$ be a balanced coloring of $P$. Since $I$ is an upper order ideal, $P\setminus I$ is a lower order ideal, and  $C^*(P,\mc{A},F,c)$ decomposes as a direct sum $C^*(P,\mc{A},F,c)=C^*(I,\mc{A},F|_I,c|_I)\oplus C^*(P\setminus I,\mc{A},F|_{P\setminus I},c|_{P\setminus I})$ where $C^*(I,\mc{A},F|_I,c|_I)$ is a subcomplex of $C^*(P,\mc{A},F,c)$. Thus we get a short exact sequence of complexes
\[
0
\to C^*(I,\mc{A},F|_I,c|_I)
\rightarrow C^*(P,\mc{A},F,c)
\rightarrow C^*(P\setminus I,\mc{A},F|_{P\setminus I},c|_{P\setminus I})
\to 0
\]
leading to the desired long exact sequence on cohomology.
\end{proof}

\section{Functoriality of Cohomology}\label{categoryofpoco}

In Section \ref{functorsonposets}, we considered functors on thin posets $P$ with balanced colorings $c$ and defined cochain complexes $C^*(P,\mc{A},F,c)$  associated to covariant  functors $F:P\to\mc{A}$ where $\mc{A}$ is an abelian category. 
In this section, we consider a categorical structure on tuples $(P,\mc{A},F,c)$, under which passing to cohomology is functorial.
In order for induced maps on complexes to be defined, we must restrict the types of maps between posets we consider.

\begin{defn}
Let $f:P\to Q$ be a map of posets. Then $f$ is called an \textit{order embedding} if for any $x,y\in P$, $x\leq y$ if and only if $f(x)\leq f(y)$. An order embedding \textit{preserves cover relations} if $x\lessdot y$ implies $f(x)\lessdot f(y)$. We may also say in this case that $f$ is a \textit{cover preserving order embedding}.
\label{CPOE}
\end{defn}

Cover preserving order embeddings were considered by Wild in \cite{wild1992cover} where they give necessary and sufficient conditions guaranteeing cover preserving order embeddings into Boolean lattices. It is an immediate consequence of $f:P\to Q$ being an order embedding that $f$ is injective, and thus $f$ is an isomorphism onto $f(P)$.


\begin{lemma}
Suppose $f:P\to Q$ is a cover preserving order embedding. Then $x\lessdot y$ if and only if $f(x)\lessdot f(y)$.
\end{lemma}
\begin{proof}
Suppose $f(x)\lessdot f(y)$ in $Q$. Since $f$ is an isomorphism from $P$ to $f(P)$, we know that $f^{-1}$ is order preserving, so we must have $x\leq y$. Suppose there exists $z$ with $x<z<y$. Then we have $f(x)\leq f(z)\leq f(y)$ but $f(x)\lessdot f(y)$ so we must have $f(z)=f(x)$ or $f(z)=f(y)$ contradicting injectivity of $f$. 
\end{proof}

Given a cover preserving order embedding $\phi:P\to Q$, each cover relation $a\lessdot b$ in the image of $P$ has a unique preimage in $P$, which we denote $\phi^{-1}(a\lessdot b)$. Given a coloring $c\in\Phi_P$ (that is, $c:C(P)\to \{1,-1\}$), let $\phi(c)$ denote the coloring on $\phi(P)\subseteq Q$ defined by $\phi(c)(a\lessdot b)=c(\phi^{-1}(a\lessdot b))$. Given a coloring $d\in\Phi_Q$, define the coloring $\phi^{-1}(d)$ on $P$ by $\phi^{-1}(d)(a\lessdot b)=d\phi(a\lessdot b)$. In either case we are really just moving the coloring from $P$ to the copy of $P$ sitting inside $Q$ or vice versa. We can also restrict $d$ to $\phi(P)$ which we shall denote by $d|_{\phi(P)}$. These satisfy the following obvious properties, proofs of which are left to the reader.
\begin{lemma}\label{colorops}
Given cover preserving order embeddings $\phi:P\to Q$, $\psi:Q\to R$ and colorings $c,d\in\Phi_P$, $e,f\in\Phi_Q$, $g\in\Phi_R$, we have 
\begin{multicols}{2}
\begin{enumerate}
    \item $\phi(cd)=\phi(c)\phi(d)$
    \item $\phi^{-1}(ef)=\phi^{-1}(e)\phi^{-1}(f)$
    \item $(\psi\circ\phi)(c)=\psi(\phi(c))$
    \item $\phi^{-1}(\psi^{-1}(g))=(\psi\circ\phi)^{-1}(g)$
    \item $c\in\Phi_P^b$ if and only if $\phi(c)\in\Phi_{\phi(P)}^b$
     \item $d\in\Phi_Q^b$ implies $\phi^{-1}(d)\in\Phi_{P}^b$
     \item $c\in\Phi_P^c$ if and only if $\phi(c)\in\Phi_{\phi(P)}^c$
     \item $d\in\Phi_Q^c$ implies $\phi^{-1}(d)\in\Phi_{P}^c$
\end{enumerate}
\end{multicols}
\end{lemma}
\begin{defn}
    Define the category $\PoCo$ to have objects $(P,\mc{A},F,c)$ where $P$ is a diamond transitive poset, $\mc{A}$ is an abelian category, $F:P\to\mc{A}$ is a functor, and $c$ is a balanced coloring of $P$. 
    Morphisms from $(P,\mc{A},F,c)$ to $(Q,\mc{B},G,d)$ are tuples $(\phi,Z,\eta,b)$ where $\phi:P\to Q$ is a cover preserving order embedding, $Z$ is an exact functor from $\mc{A}$ to $\mc{B}$, $\eta$ is a natural transformation from $ZF$ to $G\phi$, that is, a collection of morphisms $\{\eta_a:ZF(a)\to G\phi(a) \}$ indexed over $a\in P$ such that the following diagrams commute:
        \begin{center}
            \begin{tikzcd}
                ZF(a)\arrow[r,"\eta_a"]\arrow[d,"ZF(a\leq b)"']&G\phi(a)\arrow[d,"G\phi(a\leq b)"]\\
                ZF(b)\arrow[r,"\eta_b"']&G\phi(b)
            \end{tikzcd}
        \end{center}
    and $b$ is a central coloring of $P$ such that \begin{equation}
        \label{centralmorphism}
        \phi(cb)=d|_{\phi(P)}.
    \end{equation}
    Compose the morphisms $(\phi,Z,\eta,b):(P,\mc{A},F,c)\to (Q,\mc{B},G,d)$ and $(\psi,Y,\epsilon,b^\prime):(Q,\mc{B},G,d)\to (R,\mc{C},H,e)$ via 
    \[(\psi,Y,\epsilon,b^\prime)\circ (\phi,Z,\eta,b)=(\psi\circ \phi, Y\circ Z, (\epsilon \circ_h1_\phi)\circ_v(1_Y\circ_h\eta),\phi^{-1}(b^\prime)b).\]
    The identity morphism on the object $(P,\mc{A},F,c)$ is $(1_P,1_\mc{A},1_F,1):(P,\mc{A},F,c)\to (P,\mc{A},F,c)$ where $1_P:P\to P$ denotes the identity map, $1_\mc{A}$ denotes the identity functor, $1_F:F\to F$ denotes the identity natural transformation, and $1$ is the identity coloring. 
\end{defn}

\begin{remark}
One should verify that the identity truly acts as identity with respect to this composition and that $\psi\phi(c\phi^{-1}(b^\prime)b)=e|_{\psi\phi(P)}$. Verifying associativity is somewhat less straightforward so we show that in detail below.
\end{remark}
The symbols $\circ_h$ and $\circ_v$ denote the horizontal and vertical composition of natural transformations respectively (see \cite[Section 3.1]{wehrheim2016floer} or \cite[Section 1]{khovanov2010categorifications}). 
We will sometimes use the notation $\alpha:F\to G:\mc{A}\to\mc{B}$ to indicate that $\eta$ is a natural transformation between $F$ and $G$ which are functors between the categories $\mc{A}$ and $\mc{B}$.
Recall the formulas:
\[(\gamma\circ_v\alpha)_M=\gamma_M\circ\alpha_M \ \ \ \ \ \ \ \ \ (\beta\circ_h\alpha)_M=\beta_{G_1(M)}\circ F_2(\alpha_M)\]
where $\alpha:F_1\to G_1:\mc{A}\to\mc{B}$, $\beta:F_2\to G_2:\mc{B}\to\mc{C}$, and $\gamma:G_1\to H_1:\mc{A}\to\mc{B}$. 
Given morphisms
\begin{center}
\begin{tikzcd}
(P,\mc{A},F,c)\arrow[r,"{(\phi,Z,\eta,b)}"]&
(Q,\mc{B},G,d)\arrow[r,"{(\psi,Y,\varepsilon,b^\prime)}"]&
(R,\mc{C},H,e)\arrow[r,"{(\tau,X,\chi,b^{\prime\prime})}"]&
(S,\mc{D},I,f)
\end{tikzcd}
\end{center}
associating one way, we have 
\begin{align*}
&(\tau,X,\chi,b^{\prime\prime})\circ[(\psi,Y,\varepsilon,b^{\prime})\circ(\phi,Z,\eta,b)]\\
&=
\big(\tau,X,\chi,b^{\prime\prime})\circ(\psi\circ\phi,Y\circ Z,(\epsilon \circ_h1_\phi)\circ_v(1_Y\circ_h\eta),\phi^{-1}(b^\prime)b\big) \\
&=
\bigg(\tau\circ(\psi\circ\phi),X\circ (Y\circ Z),\\
&(\chi\circ_h1_{\psi\circ\phi})\circ_v\big(1_{X}\circ_h[(\epsilon \circ_h1_\phi)\circ_v(1_Y\circ_h\eta)]\big),(\psi\circ\phi)^{-1}(b^{\prime\prime})\phi^{-1}(b^\prime)b \bigg)
\end{align*}

and associating the other way, we have

\begin{align*}
&[(\tau,X,\chi,b^{\prime\prime})\circ(\psi,Y,\varepsilon,b^{\prime})]\circ(\phi,Z,\eta,b)\\
&=
\big(\tau\circ\psi,X\circ Y,(\chi\circ_h 1_\psi)\circ_v (1_X\circ_h\varepsilon),\psi^{-1}(b^{\prime\prime})b^\prime\big)\circ (\phi,Z,\eta,b)\\
&=
\bigg((\tau\circ\psi)\circ\phi,(X\circ Y)\circ Z,\\
&\big([(\chi\circ_h 1_\psi)\circ_v (1_X\circ_h\varepsilon)]\circ_h 1_\phi\big)\circ_v (1_{X\circ Y}\circ_h \eta, \phi^{-1}(\psi^{-1}(b^{\prime\prime})b^\prime)b
\bigg)
\end{align*}

Associativity in the third component follows from a straightforward computation  making use of the facts $1_{\psi\circ\phi}=1_\psi\circ_h1_\phi$ and $(\alpha\circ_h1_F)\circ_v(\beta\circ_h1_F)=(\alpha\circ_v\beta)\circ_h1_F$ for any functor $F$, and any natural transformations $\alpha,\beta$. Diagrammatically, either side of the equality is given in Figure \ref{assocfig}. Associativity in the fourth component follows from the properties listed in Lemma \ref{colorops}.

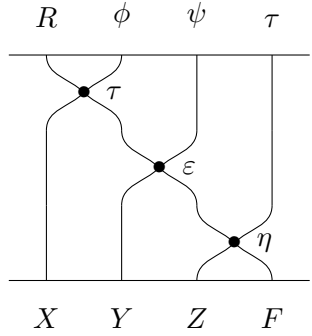
\begin{figure}
\centering
\begin{tikzpicture}
\node (R) at (.5,3){};
\node (phi) at (1.5,3){};
\node (psi) at (2.5,3){};
\node (tau) at (3.5,3){};
\node (X) at (.5,0){};
\node (Y) at (1.5,0){};
\node (Z) at (2.5,0){};
\node (F) at (3.5,0){};
\draw(0,0)--(4,0);
\draw (0,3)--(4,3);

\node[below=.1cm of X] {$X$};
\node[below=.1cm of Y] {$Y$};
\node[below=.1cm of Z] {$Z$};
\node[below=.1cm of F] {$F$};

\node[above=.1cm of R] {$R$};
\node[above=.1cm of tau] {$\tau$};
\node[above=.1cm of psi] {$\psi$};
\node[above=.1cm of phi] {$\phi$};

\node at (3,.5){\usebox\etatrans};
\draw (1.5,0)--(1.5,1);
\draw (.5,0)--(.5,1);

\draw (.5,1)--(.5,2);
\node at (2,1.5){\usebox\epsilontrans};
\draw (3.5,1)--(3.5,2);

\node at (1,2.5){\usebox\tautrans};
\draw (2.5,2)--(2.5,3);
\draw (3.5,2)--(3.5,3);
\end{tikzpicture}
\caption{The diagram shown above represents either way of associating a composition of three morphisms in $\PoCo$.}
\label{assocfig}
\end{figure}

\begin{remark}
The data of the central coloring $b$ is not really needed in the morphisms since it is determined by the corresponding balanced colorings. That is, given a morphism $(\phi,Z,\eta,b):(P,\mc{A},F,c)\to (Q,\mc{B},G,d)$, it follows that $b=c\phi^{-1}(d)$. This follows from simply by applying $\phi^{-1}$ to both sides equation (\ref{centralmorphism}) and then multiplying by $c$ using the fact that $c^2=1$.
\end{remark}


\begin{defn}
Define the category $\Abel$ whose objects are pairs $(G,\mc{A})$ where $\mc{A}$ is an abelian category and $G$ is an object in $\mc{A}$. Morphisms are pairs $(f,Z):(G,\mc{A})\to(H,\mc{B})$ where $Z:\mc{A}\to\mc{B}$ is a functor and $f:ZG\to H$ is a morphism in $\mc{B}$. Morphisms $(f,Z):(G,\mc{A})\to (H,\mc{B})$ and $(g,Y):(H,\mc{B})\to (K,\mc{C})$ are composed via $(g,Y)\circ (f,Z)=(gYf,YZ)$.
\end{defn}

\begin{defn}
Define the category $\Cabel$ whose objects are pairs $(C,\mc{A})$, where $C$ is a bounded cochain complex over the Abelian category $\mc{A}$. Morphisms are pairs $(f,Z):(C,\mc{A})\to (D,\mc{B})$ where $Z:\mc{A}\to\mc{B}$ is a functor, and $f:ZC\to D$ is a chain map. Morphisms $(f,Z):(C,\mc{A})\to (D,\mc{B})$ and $(g,Y):(D,\mc{B})\to (E,\mc{C})$ are composed via $(g,Y)\circ (f,Z)=(gYf,YZ)$.
\end{defn}

\begin{defn}
Define the category $\grAbel$ whose objects are pairs $(G,\mc{A})$ where $\mc{A}$ is an abelian category and $G$ is an object in $\mc{A}$ with a specified decomposition $G=\oplus_{i\in\Z}G^i$ where each $G^i$ is an object of $\mc{A}$. Morphisms are pairs $(f,Z):(G,\mc{A})\to(H,\mc{B})$ where $Z:\mc{A}\to\mc{B}$ is a functor and $f$ is a collection of morphisms $f^i:ZG^i\to H^i$ in $\mc{B}$ for $i\in\Z$. Morphisms $(f,Z):(G,\mc{A})\to (H,\mc{B})$ and $(g,Y):(H,\mc{B})\to (K,\mc{C})$ are composed via $(g,Y)\circ (f,Z)=(gYf,YZ)$ where $(gYf)^i=g^iYf^i$.
\end{defn}

\begin{lemma}
\label{homologyfunctor}
The association $(C,\mc{A})\mapsto (H^*(C),\mc{A})$, $(f,Z)\mapsto (H^*(f),Z)$ defines a functor from $\Cabel$ to $\grAbel$.
\end{lemma}
\begin{proof}
Given an identity morphism $(1_C,1_\mc{A})$ in $\Cabel$, the associated morphism in $\grAbel$ is $(H^*(1_C),1_\mc{A})=(1_C,1_\mc{A})$ by functoriality of taking homology. Given morphisms $(f,Z):(C,\mc{A})\to (D,\mc{B})$ and $(g,Y):(D,\mc{B})\to (E,\mc{C})$ in $\Cabel$, the morphisms associated to the composition $(gYf,YZ)$ is $(H^*(gYf),YZ)$. On the other hand, performing the composition $(H^*(g),Y)\circ (H^*(f),Z)$ in $\grAbel$, we obtain $(H^*(g)YH^*(f),YZ)$ so it suffices to show $H^*(g)YH^*(f)=H^*(gYf)$. By functoriality of $H^*$, we have $H^*(gYf)=H^*(g)H^*(Yf))$ and by exactness of $Y$ we have $H^*(Yf)=YH^*(f)$.
\end{proof}
 
Given a morphism $(\phi,Z,\eta,b):(P,\mc{A},F,c)\to (Q,\mc{B},G,d)$ in $\PoCo$, define the induced map $C^*(\phi,Z,\eta,b):ZC^*(P,\mc{A},F,c)\to C^*(Q,\mc{B},G,d)$ as follows. 
Given $A\in F(x)$ define 
    \begin{equation}
        C^*(\phi,Z,\eta,b)(ZA)=\gr^b(x)\eta_x(ZA).
        \label{chainmap}
    \end{equation}
Notice that $C^*(1_P,1_{\mc{A}},1_F,1)$ is exactly the isomorphism $C^*(P,F,c_1)\to C^*(P,F,c_2)$ defined in the proof of Theorem \ref{balindept}. 

\begin{lemma} As defined in equation \ref{chainmap}, $C^*(\phi,Z,\eta,b)$ is a chain map from $ZC^*(P,\mc{A},F,c)$ to $C^*(Q,\mc{B},G,d)$.\end{lemma}
\begin{proof} In order to check this is a chain map, we must have, for any $x\lessdot y$ in $P$,
$$d\big(\phi(x)\lessdot \phi(y)\big)G\big(\phi(x)\lessdot \phi(y)\big)[\gr^b(x)\eta_x(ZA)]=\gr^b(y)\eta_y [c(x\lessdot y)ZF(x\lessdot y)(ZA)].$$ 
By naturality, we have $G\big(\phi(x)\lessdot\phi(y)\big)\eta_x(ZA)=\eta_yZF(x\lessdot y)(ZA)$, so it remains only to show that $d\big(\phi(x)\lessdot\phi(y)\big)\gr^b(x)=\gr^b(y)c(x\lessdot y)$.
But $d|_{\phi(P)}=\phi(cb) $, so 
    \begin{align*}
        d\big(\phi(x)\lessdot \phi(y)\big)\gr^b(x)
        &=\phi(cb)\big(\phi(x)\lessdot \phi(y)\big)\gr^b(x)\\
        &=c(x\lessdot y)b(x\lessdot y)\gr^b(x)\\
        &=c(x\lessdot y)\gr^b(x)\gr^b(y)\gr^b(x)\\
        &=c(x\lessdot y)\gr^b(y). \qedhere
    \end{align*}
\end{proof}

\begin{remark}
A morphism $(\phi,Z,\eta,b)$ thus induces a chain map \[C^*(\phi,Z,\eta,b):ZC^*(P,\mc{A},F,c)\to C^*(Q,\mc{B},G,d)\] which in turn (utiliying exactness of $Z$) induces a map on homology which we denote $H^*(\phi,Z,\eta,b):ZH^*(P,\mc{A},F,c)\to H^*(Q,\mc{B},G,d)$.
\end{remark}

\begin{theorem} 
    The association \[(P,\mc{A},F,c)\mapsto (C^*(P,\mc{A},F,c),\mc{A}) \ \ \ \ \ \ \ \ \ (\phi,Z,\eta,b)\mapsto (C^*(\phi,Z,\eta,b),Z)\] defines a functor from $\PoCo$ to the category $\Cabel$. 
    \label{functortocabel}
\end{theorem}

\begin{proof}
    By construction, $C^*(1_P,1_\mc{A}, 1_F,1)$ is the identity chain map on $C^*(P,\mc{A},F,c)$. We must show the association respects composition. Suppose we have 
    \begin{center}
\begin{tikzcd}
(P,\mc{A},F,c)\arrow[r,"{(\phi,Z,\eta,b)}"]&
(Q,\mc{B},G,d)\arrow[r,"{(\psi,Y,\varepsilon,b^\prime)}"]&
(R,\mc{C},H,e)
\end{tikzcd}
\end{center}
By definition of composition in $\PoCo$,
    \[(\psi,Y,\epsilon,b^\prime)\circ (\phi,Z,\eta,b)=(\psi\circ \phi, Y\circ Z, (\epsilon \circ_h1_\phi)\circ_v(1_Y\circ_h\eta),\phi^{-1}(b^\prime)b)\]
and by definition of $C^*$, given $A\in F(x)$ for some $x\in P$, 
    \[C^*[(\psi,Y,\epsilon,b^\prime)\circ (\phi,Z,\eta,b)](YZA)=
    \gr^{\phi^{-1}(b^\prime)b}(x)[(\epsilon \circ_h1_\phi)\circ_v(1_Y\circ_h\eta)]_x(YZA).\]
On the other hand,
    \[C^*(\psi,Y,\varepsilon,b^\prime)\circ YC^*(\phi,Z,\eta,b)=
    \gr^b(x)\gr^{b^\prime}(\phi(x))\varepsilon_{\phi(x)}(Y\eta_x(YZA)).\]
It suffices to show the following two equalities
\begin{equation}
    \gr^{\phi^{-1}(b^\prime)b}(x)= \gr^b(x)\gr^{b^\prime}(\phi(x)),
    \label{greqn}
\end{equation}
\begin{equation}
    [(\epsilon \circ_h1_\phi)\circ_v(1_Y\circ_h\eta)]_x(YZA)=\varepsilon_{\phi(x)}(Y\eta_x(YZA)).
    \label{functeqn}
\end{equation}
Equation (\ref{greqn}) follows from Propositions \ref{greedyprod} and \ref{greedyinv}. Equation (\ref{functeqn}) follows from the following computation.
\begin{align*}
     [(\epsilon \circ_h1_\phi)\circ_v(1_Y\circ_h\eta)]_x(YZA)
     &=(\varepsilon\circ_h1_\phi)_x\circ(1_Y\circ_h\eta)_x(YZA)\\
     &=(\varepsilon\circ_h1_\phi)_x(Y\eta_x(YZA))\\
     &=\varepsilon_{\phi(x)}(Y\eta_x(YZA)). \qedhere
\end{align*}
\end{proof}

\begin{remark}
We will make a slight abuse of notation and denote the functor in Theorem \ref{functortocabel} by $C^*$.
\end{remark}

\begin{corollary}\label{homfunctor}
    The association \[(P,\mc{A},F,c)\mapsto (H^*(P,\mc{A},F,c),\mc{A})\] \[(\phi,Z,\eta,b)\mapsto (H^*(\phi,Z,\eta,b),Z)\] defines a functor from $\PoCo$ to $\grAbel$.
\end{corollary}

\begin{remark}
We will continue abusing notation and denote the functor in Corollary \ref{homfunctor} by $H^*$.
\end{remark}

Consider the full subcategory $\PoCo(P)$ of $\PoCo$ consisting only of objects of the form $(P,\mc{A},F,c)$ where $P$ is fixed. The functor $C^*$ from Theorem \ref{functortocabel} restricts to a functor  $C^*:\PoCo(P)\to\Cabel$ (also denoted $C^*$ by another slight abuse of notation). Given a fixed central coloring $b$, consider also the functor $C^*_{b}:\PoCo(P)\to\Cabel$ given by 
\[(P,\mc{A},F,c)\mapsto (C^*(P,\mc{A},F,cb),\mc{A})\] 
\[(\phi,Z,\eta,b^\prime)\mapsto (C^*(\phi,Z,\eta,b^\prime),Z)\] 
where the action on morphisms is defined via
\begin{center}
\begin{tikzcd}
ZC^*(P,\mc{A},F,cb)
\arrow[dotted, d]
\arrow[r]
&C^*(P,\mc{B},G,db)\\
ZC^*(P,\mc{A},F,c)
\arrow[r,dotted,]
&C^*(P,\mc{B},G,d)\arrow[u,dotted]
\end{tikzcd}
\end{center}

where the left and right vertical dotted arrows correspond to maps $ZC^*(1,1,1,b)$ and $C^*(1,1,1,b)$ respectively, and the horizontal dotted arrow is the map $C^*(\phi,Z,\eta,b^\prime)$. Actually since both the vertical maps correspond simply to multiplcation by $\gr^b(x)$, we have $C^*_b(\phi,Z,\eta,b^\prime)=C^*(\phi,Z,\eta,b^\prime)$. Therefore functoriality for $C^*_b$ follows from that of $C^*$. Consider the collection of isomorphisms \[\eta_0=\{C^*(1,1,1,b):C^*(P,\mc{A},F,c)\to C^*(P,\mc{A},F,cb)\ | \ (P,\mc{A},F,c)\in\PoCo(P)\}.\] 

\begin{lemma}
The collection $\eta_0$ of isomorphisms forms a natural transformation from $C^*$ to $C^*_b$. Thus the isomorphisms in Theorem \ref{balindept} are natural.
\label{naturality}
\end{lemma}

\begin{proof}
Consider a morphism $(\phi,Z,\eta,b^\prime):(P,\mc{A},F,c)\to (P,\mc{B},G,d) $ in $\PoCo(P)$. We must show that the following diagram commutes
\begin{center}
\begin{tikzcd}
C^*(P,\mc{A},F,c)
\arrow[d,"(\eta_0)_{(P,\mc{A},F,c)}"']
\arrow[r]
&C^*(P,\mc{B},G,d)\arrow[d,"(\eta_0)_{(P,\mc{B},G,d)}"]\\
C^*(P,\mc{A},F,cb)
\arrow[r]
&C^*(P,\mc{B},G,db)
\end{tikzcd}
\end{center}

where the horizontal maps are both equal to $C^*(\phi,Z,\eta,b^\prime)$. Thus going around the square either way is simply multiplication by $\gr^b(x)$.
\end{proof}

\begin{remark}
Of course, we could repeat all of the definitions in this section instead for contravariant functors $P\to\mc{A}$ (equivalently, functors $P^\op\to\mc{A}$), the only difference being that we obtain chain complexes (as opposed to cochain complexes) and homology (as opposed to cohomology). Doing so results in a category $\PoHo$ with all of the analogous constructions and properties as we saw for $\PoCo$, and the $\Hom$ functor defines a duality of categories $\PoHo_P\cong\PoCo_{P}^\op$.
\end{remark}

\section{Categorification via Functors on Thin Posets}\label{seccateg}

This section begins with some background information on Grothendieck groups, which can be found in \cite{weibel2013k}. We will state these results without proof and refer the reader to the literature for details. We then show how to categorify certain ring elements, called rank alternators, with cohomology theories obtained from functors on thin posets.

Let $\mc{A}$ be an abelian category, and let $\mc{C}^b(\mc{A})$ denote the category of bounded cochain complexes $C$ in $\mc{A}$, that is, such that $C^i=0$ for all but finitely many $i$. Note that $\mc{C}^b(\mc{A})$ is also an abelian category. Let $K_0(\mc{A})$ denote the \textit{Grothendieck group} of $\mc{A}$, that is, the free abelian group generated by the set of isomorphism classes $\{[X] \ | \ X\in\Ob(\mc{C})\}$ modulo the relations $[X]=[X^\prime]+[X^{\prime\prime}]$ for each short exact sequence $0\to X^\prime\to X\to X^{\prime\prime}\to 0$. If $\mc{A}$ additionally has a monoidal structure $\otimes$, then $K_0(\mc{A})$ inherits the structure of a ring, with product $[X]\cdot[Y]=[X\otimes Y]$.
\begin{theorem}[{\cite[Theorem 9.2.2]{weibel2013k}}]
\label{eulercharacteristic}
Let $\mc{A}$ be a monoidal abelian category. The Euler characteristic
$$\chi:K_0(\mc{C}^b(\mc{A}))\to K_0(\mc{A})\ \ \ \ \ \ \ \ \ \ \ \  [C]\mapsto \sum_{i\in\Z} (-1)^i[H^i(C)]$$
defines an isomorphism of rings. Furthermore, we have
\begin{equation}
\sum_{i\in\Z} (-1)^i[H^i(C)]=\sum_{i\in\Z} (-1)^i[C^i]
\end{equation}
\end{theorem}

From now on, we identify $K_0(\mc{C}^b(\mc{A}))$ with $K_0(\mc{A})$ via the isomorphism in Theorem \ref{eulercharacteristic}.

\begin{ex}\label{euler}
Let $R$ be a principal ideal domain. Let $\Rmod$ denote the (abelian) category of finitely generated (left) $R$ modules. There is an isomorphism of rings:
$$K_0(\Rmod)\cong \Z \ \ \ \ \ \ \ \ \ \ \ \  [M]\mapsto \rk M$$
(see for example \cite[Section 2]{weibel2013k}). In light of Theorem \ref{eulercharacteristic} we have the following identification of rings:
\[K_0(\mc{C}^b(\Rmod))\cong \Z \ \ \ \ \ \ \ \ \ \ \ \  [C]\mapsto \sum_{i\in\Z}(-1)^i\rk C^i.\]
\end{ex}

\begin{ex}\label{gradedeuler}
Let $R$ be a principal ideal domain. Let $\Rgmod$ denote the (abelian) category of finitely generated graded (left) $R$ modules. Given a graded module $M=\oplus_{i\in\Z}M^i\in\Rgmod$, the \textit{graded rank} of $M$ is $\qrk M=\sum_{i\in\Z}q^i\rk M^i$. Let $\Z[q,q^{-1}]$ denote the ring of Laurent polynomials in the variable $q$. There is an isomorphism of rings:
$$K_0(\Rgmod)\cong \Z[q,q^{-1}] \ \ \ \ \ \ \ \ \ \ \ \  [M]\mapsto \qrk M=\sum_{i\in\Z} q^i\rk M^i$$
(see \cite[Example 7.14]{weibel2013k}).
In light of Theorem \ref{eulercharacteristic} we have the following identification of rings:
$$K_0(\mc{C}^b(\Rgmod))\cong \Z[q,q^{-1}] \ \ \ \ \ \ \ \ \ \ \ \  [C]\mapsto \sum_{i\in\Z}(-1)^i\qrk C^i=\sum_{i,j\in\Z}(-1)^i q^j \rk C^{i,j}.$$
The map $[C]\mapsto \sum_{i,j\in\Z}(-1)^i q^j \rk C^{i,j}$ is usually referred to as the \textit{graded Euler characteristic} (see for example \cite{khovanov1999categorification} or \cite{bar2002khovanov}).
\end{ex}

\begin{theorem}\label{pococateg}
Let $\mc{A}$ be a monoidal abelian category, $P$ a balanced colorable thin poset, and $F:P\to\mc{A}$ a functor. Then in $K_0(\mc{C}^b(\mc{A}))$ we have the following equality:
\[[H^*(P,\mc{A},F)]=\sum_{x\in P}(-1)^{\rk(x)}[F(x)].\]
\end{theorem}
\begin{proof}
We identify $H^*(P,\mc{A},F)$ in $\mc{C}^b(\mc{A})$ as a chain complex with all differentials being zero maps. Then by Theorem \ref{eulercharacteristic}, we have $[H^*(P,\mc{A},F)]=[C^*(P,\mc{A},F,c)]$ where $c$ is a balanced coloring of $P$. Since $C^i(P,\mc{A},F,c)=\oplus_{\rk(x)=i}F(x)$, the desired equality follows from Theorem \ref{eulercharacteristic}.
\end{proof}
Theorem \ref{pococateg} provides the following categorification technique. In general, to categorify an element $g$ of a ring $R$, one should find a monoidal abelian category $\mc{A}$ with an isomorphism of rings $\phi:K_0(\mc{A})\to R$, and an object $G\in\mc{C}^b(\mc{A})$ such that $[G]=g$ under the identification $\phi\circ\chi$ of $K_0(\mc{C}^b(\mc{A}))$ with $R$. In this case we say that $G$ categorifies $g$.  

\begin{defn}
Suppose $P$ is a ranked poset, $R$ is a ring, and consider a function $f:P\to R$. The \textit{rank alternator} for $f$ is
\begin{equation}\label{rankalternator}
    \RGF(P,R,f)=\sum_{x\in P}(-1)^{\rk(x)}f(x).
\end{equation}
\end{defn}

\begin{remark}
As mentioned in the introduction, such expressions (\ref{rankalternator}) arise naturally for Eulerian posets in the process of applying M\"{o}bius inversion.
\end{remark}

\begin{theorem}
Let $P$ be a balanced colorable thin poset, $R$ a ring, and $f:P\to R$. Suppose $\mc{A}$ is a monoidal abelian category equipped with an isomorphism  of rings $K_0(\mc{A})\cong R$. 
Suppose $F:P\to\mc{A}$ is a functor with the property that for each $x\in P$, $[F(x)]=f(x)$. Then $H^*(P,\mc{A},F)$ categorifies $\RGF(P,R,f)$ in the sense that $[H^*(P,\mc{A},F)]=\RGF(P,R,f)$ in $K_0(\mc{A})\cong R$.
\label{categorify}
\end{theorem}
\begin{proof}
This follows immediately from Theorem \ref{pococateg}.
\end{proof}

We now look at some examples of interesting algebraic objects which arise as rank alternators over thin posets. Some of these examples have existing categorifications via the process outlined in Theorem \ref{categorify} and others do not. We begin with some known examples which are special cases of Theorem \ref{categorify}.

\begin{ex} Let $X$ be a regular CW complex, and let $f_i(X)$ denote the number of cells in $X$ of dimension $i$. Let $\bar{\mc{F}}(X)$ denote the face poset of $X$, with the empty face deleted. The Euler characteristic $\chi(X)$ of $X$ is:
\[\chi(X)=\sum_{i\in\Z}(-1)^i f_i(X)=\sum_{F\in\bar{\mc{F}}(X)}(-1)^{|F|}.\]
Let $1_{\Rmod}:\bar{\mc{F}}(X)\to\Rmod$ denote the contravariant functor sending each $F\in\bar{\mc{F}}(X)$ to $R$ and each cover relation to the identity map. Let $c$ be a balanced coloring of $\bar{\mc{F}}(X)$ (we know one exists by Theorem \ref{cwbalanced}). Then we can form the complex $C_*(\bar{\mc{F}}(X),1_{\Rmod},c)$ with homology $H_*(\bar{\mc{F}}(X),1_{\Rmod})$ and by Theorem \ref{categorify} we have $[H_*(\bar{\mc{F}}(X),1_{\Rmod})]=\chi(X)$. Of course, this construction just recovers the definition of the cellular homology $H^\cell_*(X,R)$ of $X$ with coefficients in $R$, so we have $H_*(\bar{\mc{F}}(X),1_{\Rmod})\cong H^\cell_*(X,R)$, and it is not much harder to see that $H_*(\mc{F}(X),1_{\Rmod})\cong \tilde{H}^\cell_*(X,R)$ where the tilde indicates reduced homology. Since the order complex $\Delta(\bar{\mc{F}}(X))$ is homeomorphic to $X$ (see for example \cite[Fact A2.5.2 ]{bjorner2006combinatorics}), we find $H_*(\bar{\mc{F}}(X),1_{\Rmod})\cong H^\simp_*(\Delta(\bar{\mc{F}}(X)),R)$, where $H^\simp$ denotes the simplicial homology, and similarly one can see $H_*(\mc{F}(X),1_{\Rmod})\cong \tilde{H}^\simp_*(\Delta(\bar{\mc{F}}(X)),R)$. 


\label{constantfunctorex}
\end{ex}

\begin{ex} The Kauffman bracket of a link $L$ is a rank alternating sum over a Boolean lattice (the collection of subsets of the set of crossings $X$ of a diagram for $L$):
$$\langle L\rangle = \sum_{I\in 2^X}(-1)^{|I|}q^{|I|}(q+q^{-1})^{|D(I)|}$$
where $|D(I)|$ is the number of connected components in the Kauffman state corresponding to $I$. Khovanov homology is constructed by categorifying this formula as per the process outlined in Theorem \ref{categorify} and then applying some grading shifts to obtain a homology theory which decategorifies to the Jones polynomial (which is a rescaling of the Kauffman bracket). This example was the original motivation for the generalization \ref{categorify} given in this paper. Knot Floer homology, Ozsv\'{a}th and Szab\'{o}'s categorification of the Alexander polynomial \cite{ozsvath2004holomorphic} (done independently by Rasmussen in \cite{rasmussen2003floer}), can be expressed also as the homology coming from a certain functor on a Boolean lattice \cite{cubeknotfloer2009}, although the original construction is much more geometric. 
\end{ex}

\begin{ex}
Let $G=(V,E)$ be a graph. The chromatic polynomial $\chi_G(x)$ of $G$ counts the number of proper colorings $c:V\to [x]$ for $x\in \N$. Using an inclusion-exclusion argument, the chromatic polynomial can be expressed as
\[\chi_G(x)=\sum_{S\in 2^E}(-1)^{|S|}x^{k(S)}\]
where $k(S)$ denote the number of connected components in the spanning subgraph $S$. 
Helme-Guizon and Rong \cite{helme2005categorification} construct a functor on the Boolean lattice, producing a cohomology theory which categorifies $\chi_G(x)$.  

Similar in spirit is the chromatic symmetric function $X_G(\mathbf{x})$ defined by Stanley \cite{stanley1995symmetric} as a symmetric function generalization of the chromatic polynomial. The Stanley chromatic symmetric function also admits a state sum formula:
\[X_G(\mathbf{x})=\sum_{S\in 2^E}(-1)^{|S|}p_{\lambda(S)}\]
where $\lambda(S)$ is the integer partition encoded by the sizes of connected components in $S$, $p_{\lambda(S)}(\mathbf{x})=p_{\lambda_1}(\mathbf{x})\dots p_{\lambda_k}(\mathbf{x})$, and  $p_r(\mathbf{x})=\sum_{i=0}^\infty x_i^r$ is the power sum symmetric function.  
Sazdanovi\'{c} and Yip categorify $X_G(\mathbf{x})$ in \cite{sazdanovic2018categorification} again by constructing a functor on the Boolean lattice. 
\end{ex}

Now that we have revisited some familiar homology theories, restated in terms of Theorem \ref{categorify}, we now explore some new ideas.


\begin{ex}
    This example is inspired by the discussion in \cite{stanley1998enumerative} pages 275-276. 
    Let $L$ be a lattice and let $X$ be a subset of $L$ such that $\hat{1}\notin X$ and if $s\in L$ and $s\neq \hat{1}$ then $s\leq t$ for some $t\in X$. We can take $X=L-\{\hat{0},\hat{1}\}$ for example. By The Crosscut Theorem, \cite[Corollary 3.9.4]{stanley1998enumerative},
    $$\mu(\hat{0},\hat{1})=\sum_{k\in\Z}(-1)^kN_k$$ 
    where $N_k$ is the number of $k$-subsets of $X$ whose meet is $\hat{0}$ and $\mu$ is the M\"{o}bius function. Define $X(L)$ to be the set of all subsets of $X$ whose meet is not $\hat{0}$. Then $X(L)$ is a simplicial complex and the above equation can be rewritten as
    $$\mu(\hat{0},\hat{1})=\sum_{F\in X(L)}(-1)^{|F|}.$$
    Consider the (contravariant) constant functor $F_\Z:X(L)\to\Zmod$, sending each point to $\Z$ and each arrow to the identity map. By the discussion in Example \ref{constantfunctorex}, in the Grothendieck group we have $[H_*(X(L),\Zmod,F_\Z)]=\sum_{F\in X(L)}(-1)^{|F|}=\mu(\hat{0},\hat{1})$, and  \cite[Proposition 3.8.6]{stanley1998enumerative} states that $\mu(\hat{0},\hat{1})$ is the (reduced) Euler characteristic of the order complex $\Delta(L^\prime)$ where $L^\prime=L-\{\hat{0},\hat{1}\}$, so we also have $[H_*(\Delta(L^\prime),\Z)]=\mu(\hat{0},\hat{1})$ in the Grothendieck group, where $H_*(\Delta(L^\prime),\Z)$ denotes the simplicial homology of $\Delta(L^\prime)$. Since they both categorify $\mu(\hat{0},\hat{1})$, one may ask if these two homology theories are isomorphic. It turns out that $X(L)$ is in fact homotopy equivalent to $\Delta(L^\prime)$ (this is stated in \cite[page 276]{stanley1998enumerative}) so the answer is yes, both being isomorphic to the common simplicial homology of the two homotopy equivalent spaces $X(L)$ and $\Delta(L^\prime)$.
    \label{newex2}
\end{ex}

\begin{ex}
Any determinant can be written by definition as a rank alternating sum over the  the symmetric group $S_n$:
$$\det(x_{i,j})_{i,j=1}^n=\sum_{\pi\in S_n}(-1)^{\inv(\pi)}x_{1,\pi(1)}\dots x_{n,\pi(n)}.$$
Recall that $S_n$ has a partial order (the Bruhat order) $\Br(S_n)$ which is thin (in fact, it is a CW poset). 
In \cite{chandler2018vandermonde} the current author gives a categorification of the Vandermonde determinant as the homology of a functor $F:\Br(S_n)\to\mc{A}$ from the Bruhat order on $S_n$ to a symmetric monoidal abelian category $\mc{A}$. 
\label{newex3}
\end{ex}

\begin{ex}
Let $\Delta$ be a simplicial complex. In combinatorics, one often encounters the $f$-polynomial $f(q)=\sum_{i=0}^m f_{i-1}q^i$ and the $h$-polynomial $h(q)=\sum_{i=0}^mh_iq^i$ where $f_i$ is the number of faces of dimension $i$ (subsets in $\Delta$ with $i+1$ vertices), $m-1$ is the maximum dimension of any face in $\Delta$, and $h(q)=(1-q)^mf\left(\frac{q}{1-q}\right)$ (see for example \cite[Section 8.3]{ziegler2012lectures}). In $l^2$-topology, it is $h(-q)$ that is of interest (see for example  \cite[Section 1]{boros2010f}). We can express $h(-q)$ in a way suitable for categorification:
\begin{align*}
{h}(-q)&=(1+q)^mf\left(\frac{-q}{1+q}\right)\\
&=\sum_{i=0}^m f_{i-1}(-q)^i(1+q)^{m-i}\\
&=(1+q)^{(m-n)}\sum_{F\in\Delta} (-1)^{|F|}q^{|F|}(1+q)^{|V\setminus F|}.
\end{align*}
To categorify $h(-q)$, one can construct a functor $H:\Delta\to\Rgmod$ such that $H(F)=A^{\otimes |V\setminus F|}\{|F|\}$ for some fixed $A\in\Rgmod$ with $\qrk A=1+q$ where $\{|F|\}$ denotes a degree shift by $|F|$.
\label{newex4}
\end{ex}

\begin{ex}
 \label{constantfunctor}
 \label{newex1}
The characteristic polynomial $f_P(t)$ of a ranked poset $P$ is defined as
$$f_P(t)=\sum_{x\in P}\mu(x)t^{\rk(P)-\rk(x)}$$
where $\mu$ is the M\"{o}bius function on $P$ (see for example \cite{sagan1999characteristic} for details on characteristic polynomials and M\"{o}bius functions). For Eulerian posets, $\mu(x)=(-1)^{\rk(x)}$ and therefore if $P$ is Eulerian, 
$$f_P(t)=\sum_{x\in P}(-1)^{\rk(x)} t^{\rk(P)-\rk(x)}.$$
From Example \ref{constantfunctorex}, we see that $[H_*(P,1_\mc{A})]=\sum_{x\in P}(-1)^{\rk(x)}$ is exactly the characteristic polynomial for $P$ evaluated at 1 (for Eulerian posets). In a similar fashion, one might imagine categorifying the characteristic polynomial itself (for Eulerian posets) by choosing a functor that sends $x\mapsto T^{\otimes \rk(P)-\rk(x)}$ for some fixed $T\in\mc{A}$. If such a functor $F_T$ is constructed, then in $K_0(\mc{A})$, we have $[H_*(P,F_T)]=\sum_{x\in P}(-1)^{\rk(x)}[T]^{\rk(P)-\rk(x)}=f_P([T])$.

In fact, as pointed out in the introduction, any invariant of Eulerian posets can be expressed as a rank alternator via M\"{o}bius inversion, and thus can potentially be categorified by the construction presented in this paper. Stanley \cite{stanley1994eulerian} presents several interesting invariants of Eulerian posets including the flag h-vector, the cd-index, the local h-polynomial, and the ``Ehrhart analogue'' of the local $h$-vector of a triangulation of a simplex. Categorifications of such invariants may prove interesting and will be the subject of a future paper by the author.
\end{ex}
 
\section*{Acknowledgments}
The author would like to thank Radmila Sazdanovic, Tye Lidman, Nathan Reading, Anton Mellit, Ben Hollering, Stephen Lacina, and Mike Breen-McKay for many helpful suggestions and conversations. The author's work was supported partially by the project 
P31705 of the Austrian Science Fund.

\bibliography{bib} 

\begin{thebibliography}{}

\bibitem[\protect\astroncite{Bar-Natan}{2002}]{bar2002khovanov}
Bar-Natan, D. (2002).
\newblock On {Khovanov's} categorification of the {Jones} polynomial.
\newblock {\em Algebraic \& Geometric Topology}, 2(1):337--370.

\bibitem[\protect\astroncite{Bj{\"o}rner}{1984}]{bjorner1984posets}
Bj{\"o}rner, A. (1984).
\newblock Posets, regular {CW} complexes and bruhat order.
\newblock {\em European Journal of Combinatorics}, 5(1):7--16.

\bibitem[\protect\astroncite{Bjorner and
  Brenti}{2006}]{bjorner2006combinatorics}
Bjorner, A. and Brenti, F. (2006).
\newblock {\em Combinatorics of {Coxeter} Groups}, volume 231.
\newblock Springer Science \& Business Media.

\bibitem[\protect\astroncite{Bonanzinga and
  Matveev}{2011}]{bonanzinga2011combinatorics}
Bonanzinga, M. and Matveev, M. (2011).
\newblock Combinatorics of thin posets: application to monotone covering
  properties.
\newblock {\em Order}, 28(2):173--179.

\bibitem[\protect\astroncite{Boros et~al.}{2010}]{boros2010f}
Boros, D. et~al. (2010).
\newblock f-polynomials, h-polynomials, and {E}uler characteristics.
\newblock {\em Publicacions Matem{\`a}tiques}, 54(1):73--81.

\bibitem[\protect\astroncite{Chandler}{2018}]{chandler2018vandermonde}
Chandler, A. (2018).
\newblock A categorification of the {Vandermonde} determinant.
\newblock {\em arXiv:1811.08090}.

\bibitem[\protect\astroncite{Dancso and Licata}{2015}]{dancso2015odd}
Dancso, Z. and Licata, A. (2015).
\newblock Odd {Khovanov} homology for hyperplane arrangements.
\newblock {\em Journal of Algebra}, 436:102--144.

\bibitem[\protect\astroncite{Everitt and Turner}{2014}]{turnereveritthomotopy}
Everitt, B. and Turner, P. (2014).
\newblock The homotopy theory of {K}hovanov homology.
\newblock {\em Algebr. Geom. Topol.}, 14(5):2747--2781.

\bibitem[\protect\astroncite{Everitt and Turner}{2015}]{turnereverittcellular}
Everitt, B. and Turner, P. (2015).
\newblock Cellular cohomology of posets with local coefficients.
\newblock {\em J. Algebra}, 439:134--158.

\bibitem[\protect\astroncite{Hatcher}{2002}]{Hatcher2002topology}
Hatcher, A. (2002).
\newblock {\em Algebraic topology}.
\newblock Cambridge University Press, Cambridge.

\bibitem[\protect\astroncite{Helme-Guizon and
  Rong}{2005}]{helme2005categorification}
Helme-Guizon, L. and Rong, Y. (2005).
\newblock A categorification for the chromatic polynomial.
\newblock {\em Algebraic \& Geometric Topology}, 5(4):1365--1388.

\bibitem[\protect\astroncite{Hepworth and
  Willerton}{2015}]{hepworth2015categorifying}
Hepworth, R. and Willerton, S. (2015).
\newblock Categorifying the magnitude of a graph.
\newblock {\em arXiv preprint arXiv:1505.04125}.

\bibitem[\protect\astroncite{Hersh}{2018}]{hersh2018posets}
Hersh, P. (2018).
\newblock Posets arising as 1-skeleta of simple polytopes, the nonrevisiting
  path conjecture, and poset topology.
\newblock {\em arXiv preprint arXiv:1802.04342}.

\bibitem[\protect\astroncite{Khovanov}{2000}]{khovanov1999categorification}
Khovanov, M. (2000).
\newblock A categorification of the {Jones} polynomial.
\newblock {\em Duke Math. J.}, pages 101(3):359--426.

\bibitem[\protect\astroncite{Khovanov}{2010}]{khovanov2010categorifications}
Khovanov, M. (2010).
\newblock Categorifications from planar diagrammatics.
\newblock {\em Japanese Journal of Mathematics}, 5(2):153--181.

\bibitem[\protect\astroncite{Kozlov}{2007}]{kozlov2007combinatorial}
Kozlov, D. (2007).
\newblock {\em Combinatorial algebraic topology}, volume~21.
\newblock Springer Science \& Business Media.

\bibitem[\protect\astroncite{Marietti}{2007}]{marietti2007}
Marietti, M. (2007).
\newblock Algebraic and combinatorial properties of zircons.
\newblock {\em J. Algebraic Combin.}, 26(3):363--382.

\bibitem[\protect\astroncite{Mellit}{}]{mellit}
Mellit, A.
\newblock personal communication.

\bibitem[\protect\astroncite{Ozsv{\'a}th and
  Szab{\'o}}{2004}]{ozsvath2004holomorphic}
Ozsv{\'a}th, P. and Szab{\'o}, Z. (2004).
\newblock {Holomorphic} disks and knot invariants.
\newblock {\em Advances in Mathematics}, 186(1):58--116.

\bibitem[\protect\astroncite{Ozsv\'{a}th and
  Szab\'{o}}{2009}]{cubeknotfloer2009}
Ozsv\'{a}th, P. and Szab\'{o}, Z. (2009).
\newblock A cube of resolutions for knot {F}loer homology.
\newblock {\em J. Topol.}, 2(4):865--910.

\bibitem[\protect\astroncite{Rasmussen}{2003}]{rasmussen2003floer}
Rasmussen, J. (2003).
\newblock Floer homology and knot complements.
\newblock {\em arXiv preprint math/0306378}.

\bibitem[\protect\astroncite{Sagan}{1999}]{sagan1999characteristic}
Sagan, B. (1999).
\newblock Why the characteristic polynomial factors.
\newblock {\em Bulletin of the American Mathematical Society}, 36(2):113--133.

\bibitem[\protect\astroncite{Sazdanovic and
  Yip}{2018}]{sazdanovic2018categorification}
Sazdanovic, R. and Yip, M. (2018).
\newblock A categorification of the chromatic symmetric function.
\newblock {\em Journal of Combinatorial Theory, Series A}, 154:218--246.

\bibitem[\protect\astroncite{Stanley}{1998}]{stanley1998enumerative}
Stanley, R. (1998).
\newblock Enumerative combinatorics, {Vol}. 1, {Wadsworth} \& {Brooks}/{Cole},
  {Pacific} {Grove}, {California}, 1986.

\bibitem[\protect\astroncite{Stanley}{1994a}]{stanley1994survey}
Stanley, R.~P. (1994a).
\newblock A survey of {E}ulerian posets.
\newblock In {\em Polytopes: Abstract, convex and computational}, pages
  301--333. Springer.

\bibitem[\protect\astroncite{Stanley}{1994b}]{stanley1994eulerian}
Stanley, R.~P. (1994b).
\newblock A survey of {E}ulerian posets.
\newblock In {\em Polytopes: abstract, convex and computational ({S}carborough,
  {ON}, 1993)}, volume 440 of {\em NATO Adv. Sci. Inst. Ser. C Math. Phys.
  Sci.}, pages 301--333. Kluwer Acad. Publ., Dordrecht.

\bibitem[\protect\astroncite{Stanley}{1995}]{stanley1995symmetric}
Stanley, R.~P. (1995).
\newblock A symmetric function generalization of the chromatic polynomial of a
  graph.
\newblock {\em Advances in Mathematics}, 111(1):166.

\bibitem[\protect\astroncite{Stanley}{2009}]{stanley2009promotion}
Stanley, R.~P. (2009).
\newblock Promotion and evacuation.
\newblock {\em Electron. J. Combin.}, 16(2, Special volume in honor of Anders
  Bj\"{o}rner):Research Paper 9, 24.

\bibitem[\protect\astroncite{Sto{\v{s}}i{\'c}}{2008}]{stovsic2008categorification}
Sto{\v{s}}i{\'c}, M. (2008).
\newblock Categorification of the dichromatic polynomial for graphs.
\newblock {\em Journal of Knot Theory and Its Ramifications}, 17(01):31--45.

\bibitem[\protect\astroncite{Viro}{2004}]{viro2004khovanov}
Viro, O. (2004).
\newblock Khovanov homology, its definitions and ramifications.
\newblock {\em Fund. Math}, 184:317--342.

\bibitem[\protect\astroncite{Wachs}{2006}]{wachs2006poset}
Wachs, M.~L. (2006).
\newblock Poset topology: tools and applications.
\newblock {\em arXiv preprint math/0602226}.

\bibitem[\protect\astroncite{Wehrheim}{2016}]{wehrheim2016floer}
Wehrheim, K. (2016).
\newblock Floer field philosophy.
\newblock In {\em Advances in the Mathematical Sciences}, pages 3--90.
  Springer.

\bibitem[\protect\astroncite{Weibel}{2013}]{weibel2013k}
Weibel, C.~A. (2013).
\newblock {\em The K-book: An introduction to algebraic K-theory}, volume 145.
\newblock American Mathematical Society Providence, RI.

\bibitem[\protect\astroncite{Wild}{1992}]{wild1992cover}
Wild, M. (1992).
\newblock Cover-preserving order embeddings into boolean lattices.
\newblock {\em Order}, 9(3):209--232.

\bibitem[\protect\astroncite{Ziegler}{2012}]{ziegler2012lectures}
Ziegler, G.~M. (2012).
\newblock {\em Lectures on polytopes}, volume 152.
\newblock Springer Science \& Business Media.

\end{thebibliography}

\end{document}